\documentclass[12pt]{amsart}
\usepackage[margin=1.1in]{geometry}
\usepackage{amssymb,amsfonts,amsmath,mathtools}
\usepackage{color}
\usepackage{soul}

\usepackage{enumerate}
\usepackage{mathrsfs}
\usepackage[unicode]{hyperref}
\usepackage[capitalise]{cleveref}
\usepackage[normalem]{ulem}

\usepackage{constants}
\usepackage{bbm}

\newtheorem{theorem}{Theorem}[section]
\newtheorem{corollary}[theorem]{Corollary}

\newtheorem{proposition}[theorem]{Proposition}
\newtheorem{lemma}[theorem]{Lemma}

\newtheorem{conjecture}[theorem]{Conjecture}

\newtheorem{question*}{Question}
\newtheorem{problem*}{Problem}

\theoremstyle{definition}

\theoremstyle{remark}
\newtheorem*{remark}{Remark}

\numberwithin{equation}{section}

\crefname{figure}{Figure}{Figures}
\theoremstyle{plain}
\newtheorem*{theorem*}{Theorem}
\crefname{theorems}{Theorem}{Theorems}
\crefname{corollaries}{Corollary}{Corollaries}
\newtheorem*{corollary*}{Corollary}
\crefname{corollaries*}{Corollary}{Corollaries}
\crefname{lemma}{Lemma}{Lemmata}
\crefname{proposition}{Proposition}{Propositions}
\crefname{conjectures}{Conjecture}{Conjectures}
\newtheorem*{conjonjecture*}{Conjecture}
\crefname{conjonjectures*}{Conjecture}{Conjectures}
\crefname{definitions}{Definition}{Definitions}
\crefname{hypotheses}{Hypothesis}{Hypotheses}

\newcommand{\Z}{\mathbb{Z}}
\newcommand{\R}{\mathbb{R}}
\newcommand{\Q}{\mathbb{Q}}

\newcommand{\re}{\textup{Re}}
\newcommand{\im}{\textup{Im}}

\DeclareMathOperator{\Ad}{Ad}
\DeclareMathOperator{\As}{As}

\newcommand{\GL}{\mathrm{GL}}

\newcommand{\SO}{\mathrm{SO}}
\newcommand{\SU}{\mathrm{SU}}

\newcommand{\SL}{\mathrm{SL}}
\newcommand{\Zgp}{\mathrm{Z}}

\newcommand{\A}{\mathbb{A}}

\makeatletter
\DeclareFontFamily{U}  {MnSymbolF}{}
\DeclareSymbolFont{symbolsMN}{U}{MnSymbolF}{m}{n}
\SetSymbolFont{symbolsMN}{bold}{U}{MnSymbolF}{b}{n}
\DeclareFontShape{U}{MnSymbolF}{m}{n}{
    <-6>  MnSymbolF5
   <6-7>  MnSymbolF6
   <7-8>  MnSymbolF7
   <8-9>  MnSymbolF8
   <9-10> MnSymbolF9
  <10-12> MnSymbolF10
  <12->   MnSymbolF12}{}
\DeclareFontShape{U}{MnSymbolF}{b}{n}{
    <-6>  MnSymbolF-Bold5
   <6-7>  MnSymbolF-Bold6
   <7-8>  MnSymbolF-Bold7
   <8-9>  MnSymbolF-Bold8
   <9-10> MnSymbolF-Bold9
  <10-12> MnSymbolF-Bold10
  <12->   MnSymbolF-Bold12}{}
\DeclareMathSymbol{\tbigtimes}{\mathop}{symbolsMN}{2}
\newcommand*{\bigtimes}{%
  \DOTSB
  \tbigtimes
  \slimits@ 
}
\makeatother

\renewcommand{\tilde}{\widetilde}

\renewcommand{\bar}{\overline}

\renewcommand{\epsilon}{\varepsilon}

\renewcommand{\H}{\mathbb{H}}

\renewcommand{\pmod}[1]{\, (\mathrm{mod} {\, #1})}

\newcommand{\cO}{\mathcal{O}}

\renewcommand{\Re}{\mathrm{Re}}

\def\Res{\mathop{\mathrm{Res}}}

\newcommand{\vol}{\text{\rm vol}}

\DeclareMathAlphabet{\mathpzc}{OT1}{pzc}{m}{it}

\usepackage{soul}

\renewcommand{\pmod}[1]{\,(\mathrm{mod}\,\,#1)}

\newconstantfamily{abcon}{symbol=c}

\makeatletter
\let\@wraptoccontribs\wraptoccontribs
\makeatother

\title[New variants of arithmetic quantum ergodicity]{New variants of arithmetic quantum ergodicity}
\author{Peter Humphries}
\address{Department of Mathematics, University of Virginia, Charlottesville, VA 22904, USA}
\email{\href{mailto:pclhumphries@gmail.com}{pclhumphries@gmail.com}}
\urladdr{\href{https://sites.google.com/view/peterhumphries/}{https://sites.google.com/view/peterhumphries/}}
\author{Jesse Thorner}
\address{Department of Mathematics, University of Illinois, Urbana, IL 61801, USA}
\email{\href{mailto:jesse.thorner@gmail.com}{jesse.thorner@gmail.com}}

\begin{document}

\begin{abstract}
We establish two new variants of arithmetic quantum ergodicity.  The first is for self-dual $\GL_2$ Hecke--Maa\ss{} newforms over $\Q$ as the level and Laplace eigenvalue vary jointly. The second is a nonsplit analogue wherein almost all restrictions of Hilbert (respectively Bianchi) Hecke--Maa\ss{} cusp forms to the modular surface dissipate as their Laplace eigenvalues grow.

\end{abstract}

\thanks{The first author was supported by the National Science Foundation (grant DMS-2302079) and by the Simons Foundation (award 965056). The second author was supported by the Simons Foundation (award MP-TSM-00002484) and the National Science Foundation (DMS-2401311).}

\maketitle

\section{Introduction}
\label{sec:intro}

\v{S}nirel'man \cite{MR0402834}, Colin de Verdi{\`e}re \cite{MR818831}, and Zelditch \cite{MR916129,MR1101972} established the quantum analogue of ergodic geodesic flow on a finite volume Riemannian manifold $X$.  To be more specific, let $\Delta$ denote the Laplace--Beltrami operator on $X$, and let $(\varphi_k)_{k=1}^{\infty}$ be an orthonormal basis of real-valued square-integrable eigenfunctions of $\Delta$.  Let $\lambda_{k}$ be the Laplace eigenvalue of $\varphi_k$, so that $\Delta \varphi_k = \lambda_k \varphi_k$; we may order $(\varphi_k)_{k=1}^{\infty}$ so that $(\lambda_k)_{k=1}^{\infty}$ is monotonically nondecreasing.  Consider the probability measures $d\mu_k\coloneqq |\varphi_k|^2 \, d\mu$, where $d\mu$ is the volume form on $X$.  If the geodesic flow on the unit cotangent bundle is ergodic (which happens, for instance, when $X$ has negative curvature), then there exists a density one subsequence $\varphi_{k_n}$ along which
\[
\lim_{n\to\infty}\mu_{k_n}(A)=\frac{\mu(A)}{\mu(X)}
\]
when $A$ is a continuity set.  This has been termed {\it quantum ergodicity}. Rudnick and Sarnak \cite{RS} have conjectured that
\[
\lim_{k\to\infty}\mu_k(A)= \frac{\mu(A)}{\mu(X)}
\]
when $X$ has negative sectional curvature, which would alleviate the need to pass to a density one subsequence.  This has been termed {\it quantum unique ergodicity}.

We may alternatively view quantum unique ergodicity in the following light. Consider the product space $X \times X$, which contains a diagonally embedded copy of $X$. Laplacian eigenfunctions on $X \times X$ are of the form $\varphi_{k_1}(x_1) \varphi_{k_2}(x_2)$, where $(x_1,x_2) \in X \times X$ and $\varphi_{k_1},\varphi_{k_2}$ are Laplacian eigenfunctions on $X$. The diagonal restriction of $\varphi_{k_1}(x_1) \varphi_{k_2}(x_2)$ to $X$ yields a signed measure $\varphi_{k_1}(x) \varphi_{k_2}(x) \, d\mu(x)$ on $X$. When $k_1 = k_2 = k$, this is not just a signed measure but a probability measure $\mu_k \coloneqq |\varphi_k|^2 \, d\mu$, and quantum unique ergodicity concerns the limiting behavior of these probability measures. On the other hand, the off-diagonal signed measure $\varphi_{k_1}(x) \varphi_{k_2}(x) \, d\mu(x)$ with $k_1 \neq k_2$ may interpreted as a \emph{transition amplitude}, and Zelditch posed the question of studying the limiting behavior of these signed measures as an off-diagonal analogue of quantum unique ergodicity \cite[Problem 3.2]{Zel10}. The average limiting behavior of these signed measures has implications concerning the geodesic flow on $X$, such as whether it is ergodic or is also weak mixing \cite[Theorems 4.1 and 4.4]{Zel10}.

Let $\Gamma=\SL_2(\Z)$. The orbifold $\Gamma\backslash\H$ has attracted much attention because of its arithmetic structure.  The volume form $d\mu(z)$ is the measure $y^{-2} \, dx \, dy$ (where $z=x+iy$), and $\mu(\Gamma\backslash\mathbb{H})=\frac{\pi}{3}$.  Let $(\varphi_k)_{k=1}^{\infty}$ denote an orthonormal basis of Maa\ss{} cusp forms satisfying $\Delta\varphi_k(z) = \lambda_k\varphi_k(z)$, where $\Delta=-y^2(\frac{\partial^2}{\partial x^2}+\frac{\partial^2}{\partial y^2})$.  We may diagonalize the space of Maa\ss{} cusp forms so that we may take each $\varphi_k$ to be an eigenfunction of all Hecke operators and the involution $\varphi_k(z)\mapsto \varphi_k(-\bar{z})$.  We call such an eigenfunction a Hecke--Maa\ss{} cusp form.  We expect the cuspidal spectrum of $\Gamma\backslash\mathbb{H}$ to be simple \cite{Ste94}, so that every Maa\ss{} cusp form is a Hecke--Maa\ss{} cusp form.

Let $\mathbf{1}_B$ be the indicator function of an injective geodesic ball $B$ in $\Gamma\backslash\H$, and for a probability measure $\nu$ on $\Gamma \backslash \H$, define the discrepancy
\begin{equation}
\label{eqn:discrepancy}
D(\nu)\coloneqq \sup_{B\subseteq \Gamma\backslash\H}\Big|\nu(\mathbf{1}_B)-\frac{3}{\pi}\mu(\mathbf{1}_B)\Big|.
\end{equation}
If $D(\mu_k)\to 0$ as $k\to\infty$, as predicted by the quantum unique ergodicity conjecture, then the measures $d\mu_k$ converge in the weak-* topology to $(3/\pi) d\mu$.  The rate of decay of $D(\mu_k)$ as $k \to \infty$ then quantifies the rate of convergence. Watson's thesis \cite{Watson} shows that the generalized Lindel{\"o}f hypothesis (GLH) for certain families of $\GL_1\times\GL_3$ and $\GL_2\times\GL_3$ Rankin--Selberg $L$-functions is essentially equivalent to
\begin{equation}
\label{eqn:mukB}
\mu_k(\mathbf{1}_B)=\frac{3}{\pi}\mu(\mathbf{1}_B)+O_{B,\epsilon}(\lambda_k^{-\frac{1}{4}+\epsilon})
\end{equation}
for every fixed injective geodesic ball $B \subseteq \Gamma \backslash \H$. Young refines this by explicating the dependence on $B$ in the error term in \eqref{eqn:mukB} under the assumption of GLH \cite[Proposition 1.5]{You16}. In particular, Young conditionally proves small scale quantum unique ergodicity---the error term in \eqref{eqn:mukB} is smaller than the main term so long as $\mu(\mathbf{1}_B)$ is a little larger than $\lambda_k^{-1/6}$. Moreover, even if $\mu(\mathbf{1}_B)$ is smaller than $\lambda_k^{-1/6}$, Young obtains bounds for the error term in \eqref{eqn:mukB} uniform in $B$ that are strong enough to imply the optimal bound $D(\mu_k) \ll_{\epsilon}\lambda_{k}^{-1/4+\epsilon}$, resolving a conjecture of Luo and Sarnak \cite[p.~210]{LS} conditionally under GLH.

Unconditionally, there are no known individual bounds for $D(\mu_k)$ other than the work of Lindenstrauss \cite{Lindenstrauss} and Soundararajan \cite{Sound_QUE_Maass}, which shows that
\begin{equation}
\label{eqn:QUE_Maass}
\lim_{k\to\infty}D(\mu_k)=0	
\end{equation}
with an unspecified rate of convergence. On the other hand, unconditional bounds for $D(\mu_k)$ are quite strong {\it on average}.  Luo and Sarnak \cite[Theorem 1.5]{LS} proved a strong quantitative version of Zelditch's quantum ergodicity result, namely
\begin{equation}
\label{eqn:Luo_Sarnak}
\frac{1}{|\{\lambda_k\leq T\}|}\sum_{\lambda_{k}\leq T}D(\mu_k)^2\ll_{\epsilon}T^{-\frac{1}{21}+\epsilon}.
\end{equation}
Chebyshev's inequality then implies that for fixed positive real numbers $\alpha$ and $\beta$ satisfying $2\alpha+\beta<\frac{1}{21}$, we have that
\begin{equation}
\label{eqn:Luo_Sarnak2}
\frac{|\{\lambda_k \leq T : D(\mu_{k})\geq \lambda_{k}^{-\alpha}\}|}{|\{\lambda_k\leq T\}|} \ll T^{-\beta}.
\end{equation}

Following Kowalski, Michel, and Vanderkam \cite{KMV}, one can consider variants of quantum ergodicity or quantum unique ergodicity wherein the surface on which the Hecke--Maa{\ss} forms are defined varies instead of the Laplace eigenvalue.  Given an integer $q\geq 1$, let $\Gamma_0(q)$ be the level $q$ Hecke congruence subgroup of $\Gamma$.  Let $\varphi$ be a Hecke--Maa{\ss} newform of level $q_{\varphi}$, trivial nebentypus, and Laplace eigenvalue $\lambda_{\varphi} = \frac{1}{4} + t_{\varphi}^2>0$, so that $\varphi$ is defined on the orbifold $\Gamma_0(q)\backslash\H$; we normalize $\varphi$ such that $\int_{\Gamma_0(q) \backslash \H} |\varphi(z)|^2 \, d\mu(z) = 1$.  Since $\Gamma_0(q)$ is a finite-index subgroup of $\Gamma_0(1)=\Gamma$, it follows that for each fundamental domain $\mathcal{F}$ of $\Gamma\backslash\H$, there exists a fundamental domain of $\Gamma_0(q)\backslash\H$ containing $\mathcal{F}$.

Let $C_b(\Gamma\backslash\mathbb{H})$ be the set of bounded, continuous functions $H\colon \Gamma\backslash\mathbb{H}\to\mathbb{C}$.  If $H\in C_b(\Gamma\backslash\mathbb{H})$, then the pushforward to $\Gamma\backslash\mathbb{H}$ of the $L^2$-mass of $\varphi$ (weighted by $H$) is the finite measure on $\Gamma \backslash \H$ given by
\begin{equation}
\label{eqn:measure_level}
\mu_{\varphi}(H)\coloneqq \int_{\Gamma_0(q_{\varphi})\backslash\mathbb{H}}|\varphi(z)|^2 H(z) \, d\mu(z).
\end{equation}
For fixed $\lambda>0$, the assertion that
\begin{equation}
\label{eqn:QUE_Maass_level}
\lim_{q\to\infty}\max_{\substack{q_{\varphi} = q \\ \lambda_{\varphi}\leq\lambda}} D(\mu_{\varphi})=0
\end{equation}
is one possible ``level-aspect'' variation of \eqref{eqn:QUE_Maass}, where the orbifold on which the Hecke--Maa{\ss} newforms are defined varies instead of the Laplace eigenvalues.  A proof of \eqref{eqn:QUE_Maass_level} appears to be beyond the reach of current methods, although such a result along a subsequence of {\it powerful} moduli follows from the work of Nelson, Pitale, and Saha \cite{NPS}.  Going one step further, one might hope to combine \eqref{eqn:QUE_Maass} and \eqref{eqn:QUE_Maass_level} and prove a ``hybrid-aspect'' result such as
\begin{equation}
\label{eqn:QUE_Maass_hybrid}
\lim_{Q\to\infty}\,\max_{\lambda_{\varphi}q_{\varphi} \in[Q,2Q]}
D(\mu_{\varphi})=0.
\end{equation}

The assertion \eqref{eqn:QUE_Maass_level} may be viewed as a form of quantum unique ergodicity in the Benjamini--Schramm limit. To explain this notion, we define Benjamini--Schramm convergence in further generality. We let $X = G / K$ be a Riemannian globally symmetric space of noncompact type, so that $G$ is a connected semisimple Lie group with finite center and $K$ is a maximal compact subgroup, and we let $(\Gamma_k)_{k = 1}^{\infty}$ be a sequence of cofinite lattices in $G$ whose covolume tends to infinity with $k$. The local injectivity radius of $X_k = \Gamma_k \backslash G/K$ about $x \in \Gamma_k \backslash G/K$ is
\[
\mathrm{InjRad}_{X_k}(x) \coloneqq \frac{1}{2} \inf_{\gamma \in \Gamma_k - \{1\}} d(x,\gamma x).
\]
The sequence of locally symmetric spaces $(X_k)_{k = 1}^{\infty}$ is said to Benjamini--Schramm converge to the symmetric space $X$ if for any $R > 0$,
\[
\lim_{k \to \infty} \frac{\vol\left(\left\{x \in X_k : \mathrm{InjRad}_{X_k}(x) \leq R\right\}\right)}{\vol(X_k)} = 0.
\]

The notion of quantum ergodicity in the Benjamini--Schramm limit (under the assumption of a uniform spectral gap for $(X_k)_{k = 1}^{\infty}$) has been investigated by many authors in several different settings \cite{ABL18,BM23,LS17,LS24,Pet23}. When $G = \SL_2(\R)$ and $K = \SO(2)$, the symmetric space $G/K$ is simply the hyperbolic plane $\H$, and the sequence of locally symmetric spaces $(\Gamma_0(q_k) \backslash \H)_{k = 1}^{\infty}$ is known to Benjamini--Schramm converge to the hyperbolic plane as $q_k \to \infty$ \cite[Corollary 2.2]{Rai17}; moreover, the locally symmetric spaces $(\Gamma_0(q_k) \backslash \H)_{k = 1}^{\infty}$ have a uniform spectral gap since the smallest positive Laplacian eigenvalue on these surfaces is at least $\frac{1}{4} - (\frac{7}{64})^2$ \cite[Appendix 2]{Kim}. Returning to \eqref{eqn:QUE_Maass_level}, we see that a result of this form encompasses a particular case of quantum unique ergodicity in the Benjamini--Schramm limit, while a weaker assertion of the form
\[
\lim_{q \to \infty} \frac{1}{|\{\lambda_{\varphi} \leq \lambda : q_{\varphi} = q\}|} \sum_{\substack{q_{\varphi} = q \\ \lambda_{\varphi}\leq\lambda}} D(\mu_{\varphi}) = 0
\]
would encompass quantum ergodicity in the Benjamini--Schramm limit. Slightly weaker results of this form concerning quantum ergodicity in the Benjamini--Schramm limit hold for more general sequences of locally symmetric spaces due to work of Le Masson and Sahlsten on $\H = \SL_2(\R) / \SO(2)$ \cite[Theorem 1.1]{LS17} and due to Brumley and Matz on $\SL_d(\R) / \SO(d)$ with $d \geq 3$ \cite[Theorem 1.1]{BM23}.

\section{Main results}
\label{subsec:quantum_variance}

In this paper, we investigate generalizations of the arithmetic quantum ergodicity results \eqref{eqn:Luo_Sarnak} and \eqref{eqn:Luo_Sarnak2}. Our first result implies \eqref{eqn:QUE_Maass_hybrid} along a density one subsequence of $\varphi$. Let $\mathpzc{F}$ be the set of $\GL_2$ Hecke--Maa{\ss} newforms $\varphi$ of weight zero and trivial nebentypus. Let $\Ad\varphi$ denote the adjoint lift of $\varphi$, as introduced in \cite{GJ}. The adjoint lift is a $\GL_3$ Hecke--Maa\ss{} newform of weight zero and trivial nebentypus.  Its arithmetic conductor $q_{\Ad \varphi}$ is a perfect square and $\sqrt{q_{\Ad \varphi}}$ divides $q_{\varphi}$, and its analytic conductor is $3q_{\Ad \varphi} (3 + 2|t_{\varphi}|)^2$. Recall the definitions in \eqref{eqn:discrepancy} and \eqref{eqn:measure_level}. We prove the following result.

\begin{theorem}
\label{thm:QE}
Let $Q\geq 1$ and $\mathpzc{F}(Q)\coloneqq \{\varphi\in\mathpzc{F}\colon \lambda_{\varphi}q_{\varphi}\in[Q,2Q]\}$.  If $\epsilon>0$, then
\[
|\{\Ad \varphi\colon \varphi\in\mathpzc{F}(Q),~D(\mu_{\varphi})\geq (\lambda_{\varphi}q_{\varphi})^{-\frac{\epsilon}{10^{12}}}\}|\ll_{\epsilon}Q^{\epsilon}.
\]
The implied constant is ineffective.
\end{theorem}

\begin{remark}
The implied constant is ineffective because the contribution from the dihedral Hecke--Maa\ss{} newforms in $\mathpzc{F}(Q)$ requires Siegel's ineffective lower bound for Dirichlet $L$-functions at $s=1$. If we restrict to the subfamily of  nondihedral Hecke--Maa{\ss} newforms, then our proof shows that the implied constant is effective.
\end{remark}

\cref{thm:QE} shows that there are very few adjoint lifts of $\varphi\in\mathpzc{F}(Q)$ such that the discrepancy $D(\mu_{\varphi})$ is large. This does not preclude the possibility that several $\varphi\in\mathpzc{F}(Q)$ have the same adjoint lift and $D(\varphi)$ is large.  Nonetheless, we can use \cref{thm:QE} to prove strong bounds for $D(\mu_{\varphi})^2$ hold on average over $\varphi\in\mathpzc{F}(Q)$.

\begin{corollary}
\label{cor:QE}
If $Q\geq 1$, then
\[
\frac{1}{|\mathpzc{F}(Q)|}\sum_{\varphi\in\mathpzc{F}(Q)}D(\mu_{\varphi})^2\ll Q^{-\frac{1}{10^{12}}}.
\]
The implied constant is ineffective.
\end{corollary}

Our second result is a nonsplit analogue of arithmetic quantum ergodicity wherein almost all restrictions of Hilbert (respectively Bianchi) Hecke--Maa\ss{} cusp forms to the modular surface dissipate as their Laplace eigenvalues grow.  This particular framework appears to be new to the literature.

Let $E = \Q \oplus \Q$ be the split quadratic algebra over $\Q$, so that $E \otimes_{\Q} \R = \R \oplus \R$, and consequently $\SL_2(E \otimes_{\Q} \R) / \SO_2(E \otimes_{\Q} \R) = \H \times \H$.  Let $\cO_E$ be the ring of integers of $E$, so that $\cO_E = \Z \oplus \Z$, and consequently $\SL_2(\cO_E) = \Gamma \times \Gamma$.  As $\H = \SL_2(\R) / \SO(2)$, the product space $\Gamma \backslash \H \times \Gamma \backslash \H$ may be identified with $\SL_2(\cO_E) \backslash \SL_2(E \otimes_{\Q} \R) / \SO_2(E \otimes_{\Q} \R)$.  A natural generalization is to replace the split quadratic algebra $E = \Q \oplus \Q$ over $\Q$ with a \emph{nonsplit} quadratic algebra over $\Q$, namely a quadratic number field. We first consider the case of a real quadratic field $E = \Q(\sqrt{D})$ with ring of integers $\cO_E$; for simplicity, we assume that $E$ has narrow class number $1$. Then $\SL_2(E \otimes_{\Q} \R) / \SO_2(E \otimes_{\Q} \R)$ is again equal to $\H \times \H$, but now it is no longer the case that $\SL_2(\cO_E)$ is equal to $\Gamma \times \Gamma$. Instead of the split product space $\Gamma \backslash \H \times \Gamma \backslash \H$, we work on the nonsplit space $\SL_2(\cO_E) \backslash \H \times \H$, and rather than working with products of Maa\ss{} cusp forms $\varphi_{k_1}(z_1) \varphi_{k_2}(z_2)$, which are Laplacian eigenfunctions on $\Gamma \backslash \H \times \Gamma \backslash \H$, we work with Hilbert Maa\ss{} cusp forms $\phi_j(z_1,z_2)$, which are Laplacian eigenfunctions on $\SL_2(\cO_E) \backslash \H \times \H$. We let $(\phi_j)_{j = 1}^{\infty}$ be an orthonormal basis of Hilbert Maa\ss{} cusp forms on $\SL_2(\cO_E) \backslash \H \times \H$ and we denote by $\lambda_{1,j}=\frac{1}{4}+t_{1,j}^2$ and $\lambda_{2,j}=\frac{1}{4}+t_{2,j}^2$ the two Laplace eigenvalues of $\phi_j$, where $t_{1,j},t_{2,j}$ are the associated spectral parameters of $\phi_j$.

In place of the transition amplitude $\varphi_{k_1}(z) \varphi_{k_2}(z) \, d\mu(z)$, which is the signed measure on $\Gamma \backslash \H$ obtained by restricting the Laplacian eigenfunction $\varphi_{k_1}(z_1) \varphi_{k_2}(z_2)$ on $\Gamma \backslash \H \times \Gamma \backslash \H$ to the diagonal embedding of $\Gamma \backslash \H$, we instead consider the signed measure
\[
d\mu_j(z) \coloneqq \phi_j(z,z) \, d\mu(z)
\]
on $\Gamma \backslash \H$ obtained by restricting the function $\phi_j(z_1,z_2)$ on $\SL_2(\cO_E) \backslash \H \times \H$ to the diagonal embedding of $\Gamma \backslash \H$. Given $H\in C_b(\Gamma\backslash\H)$, we define
\begin{equation}
\label{eqn:DjH}
\mathcal{D}_j(H)\coloneqq \int_{\Gamma \backslash \H} H(z) \, d\mu_j(z) - \frac{3}{\pi} \mu_j(\Gamma \backslash \H) \int_{\Gamma \backslash \H} H(z) \, d\mu(z).
\end{equation}
A natural nonsplit analogue of \cite[Problem 3.2]{Zel10} is to investigate the limiting behavior of these signed measures.

In order to determine in what sense one should take a limit involving the two spectral parameters $t_{1,j}$ and $t_{2,j}$, we introduce the quantity
\begin{equation}
\label{eqn:Asaicond}
C(\As \phi_j) \coloneqq (3 + |t_{1,j} + t_{2,j}|)^2 (3 + |t_{1,j} - t_{2,j}|)^2.
\end{equation}
This is the archimedean part of the analytic conductor of the Asai transfer $\As\phi_j$ of $\phi_j$, as introduced in \cite{Asa77}. We pose the following conjecture regarding the signed measures $d\mu_j$.

\begin{conjecture}
\label{conj:realQUE}
For any fixed $H \in C_b(\Gamma \backslash \H)$, we have that
\begin{equation}
\label{eqn:realQUE}
\lim_{C(\As \phi_j) \to \infty} \mathcal{D}_j(H) = 0.
\end{equation}
\end{conjecture}

This conjecture may be thought of as a nonsplit analogue of an off-diagonal variant of quantum unique ergodicity for $\Gamma \backslash \H$, where the restriction of a Hilbert Maa\ss{} cusp form $\phi_j(z_1,z_2)$ on $\SL_2(\cO_E) \backslash \H \times \H$ to the diagonal embedding of $\Gamma \backslash \H$ replaces the restriction of $\varphi_{k_1}(z_1) \varphi_{k_2}(z_2)$ on $\Gamma \backslash \H \times \Gamma \backslash \H$ to the diagonal embedding of $\Gamma \backslash \H$, where $\varphi_{k_1},\varphi_{k_2}$ are Maa\ss{} cusp forms on $\Gamma \backslash \H$.

As we show in \cref{lem:realBS}, if $\phi_j$ is additionally an eigenfunction of all the Hecke operators, then $\mu_j(\Gamma \backslash \H) = 0$ except on the rare occasion that $\phi_j$ is the base change of a Hecke--Maa\ss{} newform on $\Gamma_0(D) \backslash \H$ of level $D$ and nebentypus $\chi_D$, the primitive quadratic character modulo $D$. Thus, \cref{conj:realQUE} may be interpreted as stating that the restriction of a Hilbert Hecke--Maa\ss{} newform $\phi_j$ to $\Gamma \backslash \H$ dissipates, rather than equidistributes, as $C(\As \phi_j) \to \infty$ apart from when $\phi_j$ is a base change (cf.\ \cite[(92)]{Zel10}).

\cref{conj:realQUE} seems out of reach by current methods. When $\phi_j$ is the base change of a cuspidal \emph{holomorphic} Hecke eigenform of even weight $k$ and $H$ is a fixed Hecke--Maa\ss{} cusp form, Nelson has shown that the analogue of \eqref{eqn:realQUE} holds as $k \to \infty$ \cite[Theorem B]{Nel20}. His method is also valid when $\phi_j$ is the base change of a Hecke--Maa\ss{} newform provided one additionally assumes the generalized Ramanujan conjecture, but no longer applies when $\phi_j$ is not a base change.

A natural weakening of \cref{conj:realQUE} is the conjecture that there exists a density one subsequence $\phi_{j_n}$ for which \eqref{eqn:realQUE} holds for all $H \in C_b(\Gamma \backslash \H)$; this in turn may be thought of as a nonsplit analogue of quantum ergodicity. We prove the following result towards this, where we instead first fix a nice test function $H$ and then show that, apart from a very small number of exceptional Hilbert Hecke--Maa\ss{} cusp forms $\phi_j$, $|\mathcal{D}_j(H)|$ decays polynomially in $C(\As\phi_j)$.

\begin{theorem}
\label{thm:nonsplitrealQE}
Let $H \in C_c^{\infty}(\Gamma \backslash \H)$. Let $(\phi_j)_{j = 1}^{\infty}$ be an orthonormal basis of Hilbert Hecke--Maa\ss{} cusp forms.  For $Q \geq 1$, let $\mathpzc{F}_{\mathrm{As}}(Q) \coloneqq \{\phi_j\colon C(\As \phi_j) \in [Q,2Q]\}$. If $\epsilon > 0$, then
\[
|\{\phi_j \in \mathpzc{F}_{\mathrm{As}}(Q)\colon |\mathcal{D}_j(H)| > C(\As\phi_j)^{-\frac{\epsilon}{10^{12}}}\}| \ll_{D,H,\epsilon} Q^{\epsilon}.
\]
\end{theorem}

As we point out in \cref{lem:Asaisize}, \cref{thm:nonsplitrealQE} is nontrivial once $\epsilon < \tfrac{1}{2}$. We emphasize that in \cref{thm:nonsplitrealQE}, the density one sequence is dependent on the choice of test function $H \in C_c^{\infty}(\Gamma \backslash \H)$, in contrast with \eqref{eqn:Luo_Sarnak} or \cref{thm:QE}. One would like to overcome this via a diagonalization and approximation argument, as in \cite[Section 6]{MR916129}. Unfortunately, there is a major hindrance in our setting: $\mu_j$ is a signed measure, rather than a probability measure.  A diagonalization and approximation argument would require strong control over the total mass of the measure $|\mu_j|$, which we presently lack.

There is also an analogue of \cref{conj:realQUE} when $D < 0$, so that $E = \Q(\sqrt{D})$ is an imaginary quadratic field of class number $1$ with ring of integers $\cO_E$. In place of $\H \times \H = (\SL_2(\R) \times \SL_2(\R)) / (\SO(2) \times \SO(2))$, we work on hyperbolic three-space $\H^3 = \SL_2(\mathbb{C}) / \SU(2)$, where we identity $\H^3$ with the subspace $\{P = x + iy + jr \colon x + iy \in \H, \ r \in \R\}$ of the Hamiltonian quaternions. In place of an orthonormal basis of Hilbert Maa\ss{} cusp forms $(\phi_j)_{j = 1}^{\infty}$ on $\SL_2(\cO_E) \backslash \H \times \H$ with associated signed measures $d\mu_j(z) \coloneqq \phi_j(z,z) \, d\mu(z)$ on $\Gamma \backslash \H$, we work with an orthonormal basis of Bianchi Maa\ss{} cusp forms $(\phi_j)_{j = 1}^{\infty}$ on $\SL_2(\cO_E) \backslash \H^3$. We denote by $\lambda_j=1+4t_j^2$ the Laplace eigenvalue of $\phi_j$, where $t_j$ is the associated spectral parameter. The archimedean part of the analytic conductor of $\phi_j$ is given by $C(\phi_j) \coloneqq (3 + |t_j|)^4$. Consider the signed measures $d\mu_j(z) \coloneqq \phi_j(z) \, d\mu(z)$ on $\Gamma \backslash \H$ obtained by restricting the function $\phi_j(P)$ on $\SL_2(\cO_E) \backslash \H^3$ to $\Gamma \backslash \H$. With this alteration of $d\mu_j(z)$, we then define $\mathcal{D}_j(H)$ just as in \eqref{eqn:DjH}. We pose the following conjecture.

\begin{conjecture}
\label{conj:imQUE}
If $H \in C_b(\Gamma \backslash \H)$ is fixed, then
\begin{equation}
\label{eqn:imQUE}
\lim_{C(\phi_j) \to \infty} C(\phi_j)^{\frac{1}{8}} \mathcal{D}_j(H) = 0.
\end{equation}
\end{conjecture}

\cref{conj:imQUE} is analogous to \cref{conj:realQUE}, with two notable differences:
\begin{enumerate}
\item In \cref{conj:imQUE}, we take the limit as $C(\phi_j)$ tends to infinity, rather than $C(\As \phi_j)$; this is due to the fact that $C(\phi_j) \asymp C(\As \phi_j)$ in this setting.
\item In \cref{conj:imQUE}, we additionally insert a factor $C(\phi_j)^{1/8}$; this is due to the fact that the main term $\mu_j(\Gamma \backslash \H) \int_{\Gamma \backslash \H} H(z) \, d\mu(z)$ may be of size $C(\phi_j)^{-1/8}$ in this setting.
\end{enumerate}
If $\phi_j$ is additionally an eigenfunction of all the Hecke operators, we show that $\mu_j(\Gamma \backslash \H) = 0$ except on the rare occasion that $\phi_j$ is the base change of a Hecke--Maa\ss{} newform on $\Gamma_0(D) \backslash \H$ of weight $1$, level $-D$, and nebentypus $\chi_D$, the primitive quadratic character modulo $-D$. Thus, \cref{conj:imQUE} may be interpreted as stating that once normalized by a multiplicative factor $C(\phi_j)^{1/8}$, the restriction of a Bianchi Hecke--Maa\ss{} cusp form $\phi_j$ to $\Gamma \backslash \H$ dissipates, rather than equidistributes, as $C(\phi_j) \to \infty$ unless $\phi_j$ is a base change.

We prove the following result towards a quantum ergodicity analogue of \cref{conj:imQUE}.

\begin{theorem}
\label{thm:nonsplitimQE}
Let $H \in C_c^{\infty}(\Gamma \backslash \H)$. Let $(\phi_j)_{j = 1}^{\infty}$ be an orthonormal basis of Bianchi Hecke--Maa\ss{} cusp forms. For $Q \geq 1$, let $\mathscr{F}(Q)\coloneqq \{\phi_j \colon C(\phi_j)\in [Q, 2Q]\}$.  If $\epsilon > 0$, then
\[
\big|\big\{\phi_j \in \mathpzc{F}_{\As}(Q) \colon \big|\mathcal{D}_j(H)\big| > C(\phi_j)^{-\frac{1}{8} - \frac{\epsilon}{10^{12}}}\big\}\big| \ll_{D,H,\epsilon} Q^{\epsilon}.
\]
\end{theorem}

We have made no effort to optimize the factor of $\frac{1}{10^{12}}$ appearing in the exponents in \cref{thm:QE,thm:nonsplitrealQE,thm:nonsplitimQE}; it could be improved with more care. Nonetheless, the method of proof cannot obtain an exponent nearly as strong as the exponent $\frac{1}{21}$ appearing in Luo and Sarnak's estimate \eqref{eqn:Luo_Sarnak}. This arises from a fundamental difference in the method of proof.  In contrast to the work of Luo and Sarnak in \cite{LS}, which relies heavily on Poincar{\'e} series, the proofs of \cref{thm:QE,thm:nonsplitrealQE,thm:nonsplitimQE} rely on spectral expansions and period integral identities (generalizing work of Ichino and Watson) that relate the inner products in these expansions to Rankin--Selberg $L$-functions on the critical line.  The work of Soundararajan and the second author \cite{ST} relates bounds for $L$-functions on the critical line to the scarcity of zeros of $L$-functions near the line $\re(s)=1$.  The desired scarcity follows from zero density estimates.

\subsection*{Acknowledgements}

We thank Yao Cheng, Alexandre de Faveri, Daniel Hu, and the anonymous referees for helpful comments.

\section{Properties of \texorpdfstring{$L$}{L}-functions}
\label{sec:L-functions}


We recall some standard facts about $L$-functions arising from automorphic representations and their Rankin--Selberg convolutions.  See \cite{Brumley,GJ2,JPSS,MW,ST}.

\subsection{Standard \texorpdfstring{$L$}{L}-functions}

Let $\mathfrak{F}_n$ be the set of cuspidal automorphic representations $\pi = \pi_{\infty} \otimes \bigotimes_p' \pi_p$ of $\GL_n(\A_{\Q})$, where the restricted tensor product runs over all primes and $\pi$ is normalized so that its central character is trivial on the positive reals. Given $\pi\in\mathfrak{F}_n$, let $\widetilde{\pi}\in\mathfrak{F}_n$ be the contragredient representation and $q_{\pi}$ be the arithmetic conductor of $\pi$. The local $L$-function $L(s,\pi_{p})$ is defined in terms of the Satake parameters $\alpha_{j,\pi}(p) \in \mathbb{C}$ by
\begin{equation}
	\label{eqn:Euler_p_single}
	L(s,\pi_{p})=\prod_{j=1}^{n}(1-\alpha_{j,\pi}(p)p^{-s})^{-1}=\sum_{k=0}^{\infty}\frac{\lambda_{\pi}(p^k)}{p^{ks}}.
\end{equation}
If $p\nmid q_{\pi}$, then $\alpha_{j,\pi}(p)\neq0$ for all $j$.  If $p \mid q_{\pi}$, then there might exist $j$ such that $\alpha_{j,\pi}(p)=0$.  The standard $L$-function $L(s,\pi)$ associated to $\pi$ is of the form
\[
L(s,\pi)=\prod_{p} L(s,\pi_{p})=\sum_{n = 1}^{\infty}\frac{\lambda_{\pi}(n)}{n^s}.
\]
The Euler product and Dirichlet series converge absolutely when $\re(s)>1$.

At the archimedean place, there are $n$ Langlands parameters $\mu_{j,\pi}\in\mathbb{C}$ such that
\[
L(s,\pi_{\infty}) = \prod_{j=1}^{n}\Gamma_{\R}(s+\mu_{j,\pi}),\qquad \Gamma_{\R}(s) \coloneqq \pi^{-\frac{s}{2}}\Gamma(\tfrac{s}{2}).
\]

Let $r_{\pi}$ be the order of the pole of $L(s,\pi)$ at $s=1$: this is $0$ unless $n = 1$ and $\pi$ is trivial, in which case $L(s,\pi)$ is the Riemann zeta function, which has a simple pole at $s=1$.  The completed $L$-function $\Lambda(s,\pi) = (s(s-1))^{r_{\pi}} q_{\pi}^{s/2}L(s,\pi)L(s,\pi_{\infty})$ is entire of order 1, and there exists a complex number $W(\pi)$ of modulus 1 such that if $s\in\mathbb{C}$, then $\Lambda(s,\pi)=W(\pi)\Lambda(1-s,\widetilde{\pi})$. The analytic conductor of $\pi$ is given by
\begin{equation}
\label{eqn:analytic_conductor_def}
C(\pi,t)\coloneqq q_{\pi}\prod_{j=1}^n(3+|it+\mu_{j,\pi}|),\qquad C(\pi)\coloneqq C(\pi,0).
\end{equation}

\subsection{Rankin--Selberg \texorpdfstring{$L$}{L}-functions}
\label{subsec:RS}

Let $\pi\in\mathfrak{F}_n$ and $\pi'\in\mathfrak{F}_{n'}$.  At each prime $p$, Jacquet, Piatetski-Shapiro, and Shalika \cite{JPSS} associate to $\pi_{p}$ and $\pi_{p}'$ a local Rankin--Selberg $L$-function
\begin{equation}
\label{eqn:RS_Dirichlet_series}
L(s,\pi_{p}\times\pi_{p}')=\prod_{j=1}^{n}\prod_{j'=1}^{n'}(1-\alpha_{j,j',\pi\times\pi'}(p) p^{-s})^{-1}=\sum_{k=0}^{\infty}\frac{\lambda_{\pi\times\pi'}(p^k)}{p^{ks}}
\end{equation}
and a local conductor $q_{\pi_{p}\times\pi_{p}'}$.  If $p\nmid q_{\pi}q_{\pi'}$, then we have the equality of sets
\begin{equation}
\label{eqn:separate_dirichlet_coeffs}
\{\alpha_{j,j',\pi\times\pi'}(p)\}=\{\alpha_{j,\pi}(p)\alpha_{j',\pi'}(p)\}.
\end{equation}
The Rankin--Selberg $L$-function $L(s,\pi\times\pi')$ associated to $\pi$ and $\pi'$ and its arithmetic conductor are
\[
L(s,\pi\times\pi')=\prod_{p}L(s,\pi_{p}\times\pi_{p}')=\sum_{n=1}^{\infty}\frac{\lambda_{\pi\times\pi'}(n)}{n^s},\qquad q_{\pi\times\pi'}=\prod_{p}q_{\pi_{p}\times\pi_{p}'}.
\]
Jacquet, Piatetski-Shapiro, and Shalika associate $n'n$ complex Langlands parameters $\mu_{\pi\times\pi'}(j,j')$ to $\pi_{\infty}$ and $\pi_{\infty}'$, from which one defines
\[
L(s,\pi_{\infty}\times\pi_{\infty}') = \prod_{j=1}^{n}\prod_{j'=1}^{n'}\Gamma\Big(\frac{s+\mu_{\pi\times\pi'}(j,j')}{2}\Big).
\]

Let $r_{\pi\times\pi'}$ be the order of the pole of $L(s,\pi\times\pi')$ at $s = 1$.  By our normalization for the central characters of $\pi$ and $\pi'$, we have that $r_{\pi\times\pi'}=0$ unless $\pi'=\tilde{\pi}$, in which case $r_{\pi\times\widetilde{\pi}}=1$.  The completed $L$-function
\begin{equation}
\label{eqn:Lambdaspixpi'}
\Lambda(s,\pi\times\pi')=(s(s-1))^{r_{\pi\times\pi'}}q_{\pi\times\pi'}^{s/2}L(s,\pi\times\pi')L(s,\pi_{\infty}\times\pi_{\infty}')
\end{equation}
is entire of order 1, and there exists a complex number $W(\pi\times\pi')$ of modulus 1 such that
\[
\Lambda(s,\pi\times\pi')=W(\pi\times\pi')\Lambda(1-s,\widetilde{\pi}\times\widetilde{\pi}').
\]
As with $L(s,\pi)$, the analytic conductor of $L(s,\pi\times\pi')$ is given by
\begin{equation}
\label{eqn:analytic_conductor_def_2}
C(\pi\times\pi',t)\coloneqq q_{\pi\times\pi'}\prod_{j=1}^n \prod_{j'=1}^{n'}(3+|it+\mu_{\pi\times\pi'}(j,j')|),\qquad C(\pi\times\pi')\coloneqq C(\pi\times\pi',0).
\end{equation}
The combined work of Bushnell and Henniart \cite{BH} and Brumley \cite[Appendix]{Humphries} yields
\begin{equation}
\label{eqn:BH}
C(\pi\times\pi',t)\ll C(\pi\times\pi')(3+|t|)^{n'n},\qquad C(\pi\times\pi')\ll C(\pi)^{n'}C(\pi')^{n}.
\end{equation}

\subsection{Zeros of \texorpdfstring{$L$}{L}-functions}
\label{sec:zeros}

For $Q \geq 1$, we denote by $\mathfrak{F}_n(Q)$ the set of cuspidal automorphic representations $\pi$ of $\GL_n(\A_{\Q})$ with analytic conductor $C(\pi)$ at most $Q$. We shall show that for any fixed cuspidal automorphic representation $\pi'$, a subconvex bound for $L(\tfrac{1}{2} + it,\pi \times \pi')$ holds for \emph{most} $\pi \in \mathfrak{F}_n(Q)$ in the large $Q$ limit. The proof relies on a zero density estimate, which, for $0 \leq \sigma \leq 1$ and $T\geq 1$, concerns the count (with multiplicity)
\[
N_{\pi\times\pi'}(\sigma,T)\coloneqq |\{\rho=\beta+i\gamma\colon L(\rho,\pi\times\pi')=0,~\beta\geq\sigma,~|\gamma|\leq T\}|.
\]

\begin{lemma}
\label{lem:ZDE}
Let $n,n'\in\{1,2,3,4\}$ and $0<\epsilon\leq 1$.  If $Q \geq 1$, $\pi'\in\mathfrak{F}_{n'}(Q^{1/11})$, and $1\leq T\leq Q^{1/11}+6$, then
\[
\sum_{\pi\in\mathfrak{F}_n(Q)}N_{\pi\times\pi'}\Big(1-\frac{\epsilon}{150},T\Big)\ll_{\epsilon} Q^{\epsilon}.
\]
\end{lemma}
\begin{proof}
Let $n,n'\geq 1$ be arbitrary.  We invoke \cite[Theorem 1.1]{HumphriesThorner_ZDE} with $\mathcal{S}=\mathfrak{F}_n(Q)$, which, after rescaling $\epsilon$, states that for $0 \leq \sigma \leq 1$ and $T\geq 1$,
\[
\sum_{\pi\in\mathfrak{F}_n(Q)}N_{\pi\times\pi'}(\sigma,T)\ll_{n,n',\epsilon} \big(|\mathfrak{F}_n(Q)|^{4} \big(C(\pi') Q  T\big)^{6.15\max\{n^2,n'n\}}\big)^{1-\sigma+\frac{\epsilon}{10^6}}.
\]
The desired result now follows from our bounds on $T$ and $C(\pi')$, the bound $|\mathfrak{F}_n(Q)|\ll 
Q^{2n+1/4}$ that follows from \cite[Appendix]{BTZ}, and the restriction to $n,n'\in\{1,2,3,4\}$.
\end{proof}

We will apply \cref{lem:ZDE} to study moments of $L$-function using the following result.

\begin{lemma}
\label{lem:central_values}
Let $n,n'\geq 1$.  Let $(\pi,\pi')\in\mathfrak{F}_{n}\times\mathfrak{F}_{n'}$, $t\in\R$, and $\epsilon>0$.  For all $0\leq \alpha<1/2$, there exist effectively computable constants $\Cl[abcon]{kappa_const1}=\Cr{kappa_const1}(n)>0$ and $\Cl[abcon]{kappa_const}=\Cr{kappa_const}(n,n')>0$ such that
\begin{align*}
	\log|L(\tfrac{1}{2}+it,\pi)|&\leq\Big(\frac{1}{4}-\frac{\alpha}{10^{9}}
\Big)\log(C(\pi)(3+|t|)^{n})\\
&+\frac{\alpha}{10^7}N_{\pi}(1-\alpha,|t|+6)+2\log|L(\tfrac{3}{2}+it,\pi)|+\Cr{kappa_const1}.
\end{align*}
and
\begin{align*}
	\log|L(\tfrac{1}{2}+it,\pi\times\pi')|&\leq\Big(\frac{1}{4}-\frac{\alpha}{10^{9}}
\Big)\log(C(\pi)^{n'}C(\pi')^{n}(3+|t|)^{n'n})\\
&+\frac{\alpha}{10^7}N_{\pi\times\pi'}(1-\alpha,|t|+6)+2\log|L(\tfrac{3}{2}+it,\pi\times\pi')|+\Cr{kappa_const}.
\end{align*}
In particular, the following bounds hold:
\begin{equation}
\label{eqn:convexity}
\begin{aligned}
L(\tfrac{1}{2}+it,\pi)&\ll_{n,n'} C(\pi)^{\frac{1}{4}}(3+|t|)^{\frac{n}{4}}\\
L(\tfrac{1}{2}+it,\pi\times\pi')&\ll_{n,n',\epsilon} (C(\pi)^{n'}C(\pi')^n)^{\frac{1}{4}}(3+|t|)^{\frac{n'n}{4}+\epsilon}.
\end{aligned}
\end{equation}
\end{lemma}
\begin{proof}
For the bounds on the logarithm, it suffices for us to consider $L(\frac{1}{2}+it,\pi\times\pi')$, since all of the results for $L(\frac{1}{2}+it,\pi)$ would then follow by choosing $\pi'\in\mathfrak{F}_1$ to be trivial.  We mimic the proof of \cite[Theorem 1.1]{ST}, replacing $\pi'$ with $\pi' \otimes \left|\det\right|^{it}$.  This has the effect of adding $it$ to each Langlands parameter $\mu_{\pi\times\pi'}(j,j')$, which, after an application of \eqref{eqn:BH}, yields
\begin{align*}
\log|L(\tfrac{1}{2}+it,\pi\times\pi')|&\leq\Big(\frac{1}{4}-\frac{\alpha}{10^{9}}\Big)\log C(\pi\times\pi',t) +2\log|L(\tfrac{3}{2}+it,\pi\times\pi')|\\
&+ \frac{\alpha}{10^{7}} |\{\rho=\beta+i\gamma\colon \beta\geq1-\alpha,~|\gamma-t|\leq 6\}|+O_{n,n'}(1)\\
&\leq\Big(\frac{1}{4}-\frac{\alpha}{10^{9}}\Big)\log C(\pi\times\pi',t) +2\log|L(\tfrac{3}{2}+it,\pi\times\pi')|\\
&+ \frac{\alpha}{10^{7}} N_{\pi\times\pi'}(1-\alpha,|t|+6)+O_{n,n'}(1)\\
&\leq\Big(\frac{1}{4}-\frac{\alpha}{10^{9}}\Big)\log(C(\pi)^{n'}C(\pi')^{n} (3+|t|)^{n'n}) +2\log|L(\tfrac{3}{2}+it,\pi\times\pi')|\\
&+ \frac{\alpha}{10^{7}} N_{\pi\times\pi'}(1-\alpha,|t|+6)+O_{n,n'}(1).
\end{align*}

If $\pi'$ is trivial and $\alpha=0$, then the above estimate shows that
\[
L(\tfrac{1}{2}+it,\pi)\ll_{n}C(\pi)^{\frac{1}{4}}(3+|t|)^{\frac{n}{4}}|L(\tfrac{3}{2}+it,\pi)|^2.
\]
The bound $|L(\tfrac{3}{2}+it,\pi)|^2\ll_n 1$ follows from the bound $|\alpha_{j,\pi}(p)|\leq p^{1/2-1/(n^2+1)}$ \cite{LRS,MS}.  If $\pi'$ is nontrivial and $\alpha=0$, then the above work and \eqref{eqn:BH} show that
\[
L(\tfrac{1}{2}+it,\pi\times\pi')\ll_{n,n'}(C(\pi)^{n'}C(\pi')^n)^{\frac{1}{4}}(3+|t|)^{\frac{n'n}{4}}|L(\tfrac{3}{2}+it,\pi)|^2.
\]
Let $\epsilon>0$.  The bound
\begin{equation}
\label{eqn:Li}
|L(\tfrac{3}{2}+it,\pi\times\pi')|\ll_{n,n',\epsilon}(C(\pi)C(\pi'))^{\epsilon}
\end{equation}
follows from \cite[Theorem 2]{Li}.
\end{proof}

\begin{proposition}
\label{prop:subconvexity}
Let $n,n'\in\{1,2,3,4\}$ and $\pi'\in\mathfrak{F}_{n'}(Q^{1/11})$.  If $\epsilon>0$, then with $O_{\epsilon}(Q^{\epsilon})$ exceptions, each $\pi\in\mathfrak{F}_n(Q)$ satisfies
\[
|L(\tfrac{1}{2} + it,\pi)|\leq (C(\pi) (3+|t|)^{n})^{\frac{1}{4}-\frac{6\epsilon}{10^{12}}}\qquad\text{for all $t\in[-Q^{\frac{1}{11}},Q^{\frac{1}{11}}]$.}
\]
and
\[
|L(\tfrac{1}{2} + it,\pi\times\pi')|\leq (C(\pi)^{n'}C(\pi')^n (3+|t|)^{n'n})^{\frac{1}{4}-\frac{6\epsilon}{10^{12}}}\qquad\text{for all $t\in[-Q^{\frac{1}{11}},Q^{\frac{1}{11}}]$.}
\]
\end{proposition}
\begin{proof}
This follows immediately from \cref{lem:ZDE,lem:central_values} with $\alpha=\epsilon/150$ and \eqref{eqn:Li} (with $\epsilon$ rescaled to $\epsilon/10^{10}$).
\end{proof}

\section{Hybrid-aspect quantum ergodicity for \texorpdfstring{$\GL_2$}{GL\9040\202} Hecke--Maa\ss{} newforms}
\label{sec:QE}

Let $\varphi$ be a Hecke--Maa\ss{} newform on $\Gamma_0(q_{\varphi})\backslash\mathbb{H}$ with trivial nebentypus and Laplace eigenvalue $\lambda_{\varphi}=\frac{1}{4}+t_{\varphi}^2>0$.  Then $\varphi$ is an eigenfunction of the hyperbolic Laplacian $\Delta$, all of the Hecke operators, and the involution $T_{-1}$ sending $\varphi(z)$ to $\varphi(-\bar{z})$.  The eigenvalue $W_{\varphi}$ of $\varphi$ for $T_{-1}$ is either $1$ or $-1$, leading to the respective Fourier expansions
\begin{equation}
\label{eqn:Fourier}
\begin{aligned}
	\varphi(x+iy)&=\rho(\varphi)\sqrt{y}\sum_{m=1}^{\infty}\lambda_{\varphi}(m)K_{it_{\varphi}}(2\pi my)\cos(2\pi mx),\\
	\varphi(x+iy)&=\rho(\varphi)\sqrt{y}\sum_{m=1}^{\infty}\lambda_{\varphi}(m)K_{it_{\varphi}}(2\pi my)\sin(2\pi nx),
\end{aligned}
\end{equation}
where $\rho(\varphi)$ is a positive normalizing constant. We consider the family
\[
\mathpzc{F}(Q)\coloneqq \{\varphi\colon\lambda_{\varphi}q_{\varphi} \in [Q,2Q]\}.
\]
The estimate
\begin{equation}
\label{eqn:F(Q)_bound}
|\mathpzc{F}(Q)|\asymp Q^2
\end{equation}
follows from the work of Brumley and Mili{\'c}evi{\'c} in \cite{BM}.

Let $(\varphi_{j})_{j=1}^{\infty}$ be the sequence of Hecke--Maa\ss{} cusp forms on $\Gamma\backslash\mathbb{H}$ normalized to have Petersson norm 1, and let $E(\cdot,\frac{1}{2}+it)$ denote a real-analytic Eisenstein series.  With $z=x+iy$, let
\[
\langle f,g\rangle_{q} \coloneqq  \int_{\Gamma_0(q)\backslash\mathbb{H}}f(z)\bar{g(z)}  \,d\mu(z)
\]
be the level $q$ Petersson inner product.  For $H\in C_b(\Gamma\backslash\mathbb{H})$ and a Hecke--Maa\ss{} newform $\varphi$ on $\Gamma_0(q_{\varphi})\backslash\mathbb{H}$, we define
\[
\mu_{\varphi}(H)=\int_{\Gamma_0(q_{\varphi})\backslash\mathbb{H}}|\varphi(z)|^2 H(z)  \, d\mu(z)=\langle H,|\varphi|^2\rangle_{q_{\varphi}}.
\]
We always consider $\varphi$ to be normalized so that $\mu_{\varphi}$ is a probability measure on $\Gamma\backslash\H$.  Subject to this normalization, we take the positive constant $\rho(\varphi)$ in \eqref{eqn:Fourier} to be such that $\lambda_{\varphi}(1)=1$.

\subsection{Preliminaries}

Let $B( w ,r)$ be an injective geodesic ball on $\Gamma\backslash\mathbb{H}$ of radius $r$ centered at $ w \in \Gamma\backslash\H$, and let $\mathbf{1}_{B( w ,r)}$ be its indicator function.  We will study the discrepancy
\[
D(\mu_{\varphi})=\sup_{\substack{B( w ,r)\subseteq \Gamma\backslash\mathbb{H} \\ r>0,~ w \in\Gamma\backslash\H}}\Big|\mu_{\varphi}(\mathbf{1}_{B( w ,r)})-\frac{3}{\pi}\mu(\mathbf{1}_{B( w ,r)})\Big|.
\]
Choose $T\geq e$ and define
\begin{align*}
D^{T}(\mu_{\varphi}) &\coloneqq  \sup_{\substack{B( w ,r)\subseteq \Gamma\backslash\mathbb{H} \\  w \in \Gamma\backslash\H,~\im( w )\geq 2T}}\Big|\mu_{\varphi}(\mathbf{1}_B) - \frac{3}{\pi}\mu(\mathbf{1}_B)\Big|,\\
D_{T}(\mu_{\varphi}) &\coloneqq  \sup_{\substack{B( w ,r)\subseteq \Gamma\backslash\mathbb{H} \\  w \in \Gamma\backslash\H,~\im( w )< 2T}}\Big|\mu_{\varphi}(\mathbf{1}_B) - \frac{3}{\pi}\mu(\mathbf{1}_B)\Big|.
\end{align*}
Note that $D(\mu_{\varphi})=\max\{D_T(\mu_{\varphi}),D^T(\mu_{\varphi})\}$.  We first bound $D^T(\mu_{\varphi})$ using the work of Soundararajan \cite{Sound_QUE_Maass}.

\begin{lemma}
	\label{lem:Sound}
	Let $\varphi\in\mathpzc{F}(Q)$.  If $T\geq e$, then $D^T(\mu_{\varphi})\ll (\log T)/\sqrt{T}$.
\end{lemma}
\begin{proof}
Consider an injective geodesic ball $B( w ,r)$ with $\im( w )\geq 2T$.  Observe that
\begin{align*}
\Big|\mu_{\varphi}(\mathbf{1}_{B( w ,r)})-\frac{3}{\pi}\mu(\mathbf{1}_{B( w ,r)})\Big|\leq \mu_{\varphi}(\mathbf{1}_{B( w ,r)})+\frac{3}{\pi}\mu(\mathbf{1}_{B( w ,r)})\leq \mu_{\varphi}(\mathbf{1}_{B( w ,r)})+O(T^{-2})
\end{align*}
and
\[
\mu_{\varphi}(\mathbf{1}_{B( w ,r)})\ll \int_{\substack{|x|\leq\frac{1}{2} \\ y\geq T}}|\varphi(x+iy)|^2  \, \frac{dx \, dy}{y^2}.
\]
We expand $\varphi$ according to \eqref{eqn:Fourier} and apply Parseval's identity to obtain
\[
\int_{\substack{|x|\leq\frac{1}{2} \\ y\geq T}}|\varphi(x+iy)|^2  \, \frac{dx \, dy}{y^2}=\frac{\rho(\varphi)^2}{2}\int_1^{\infty}|K_{ir}(2\pi t)|^2\sum_{m\leq t/T}|\lambda_{\varphi}(m)|^2 \, \frac{dt}{t}.
\]

Since $\varphi$ is has trivial nebentypus, the Hecke relations \cite[(0.3)]{HL} give us
\[
\lambda_{\varphi}(m_1)\lambda_{\varphi}(m_2)=\sum_{\substack{d|\gcd(m_1,m_2) \\ \gcd(d,q)=1}}\lambda_{\varphi}\Big(\frac{m_1 m_2}{d^2}\Big).
\]
Consequently, for integers $m,m_1,m_2\geq 1$ and a prime $p$, we have the bounds
\begin{align*}
	&|\lambda_{\varphi}(p)|^2\leq 1+|\lambda_{\varphi}(p^2)|,\qquad |\lambda_{\varphi}(mp^2)|\leq|\lambda_{\varphi}(p^2)\lambda_{\varphi}(m)|+|\lambda_{\varphi}(m)|+\Big|\lambda_{\varphi}\Big(\frac{m}{p^2}\Big)\Big|,\\
	&|\lambda_{\varphi}(m_1)\lambda_{\varphi}(m_2)|\leq \sum_{d|\gcd(m_1,m_2)}\Big|\lambda_{\varphi}\Big(\frac{m_1 m_2}{d^2}\Big)\Big|,\qquad 
	|\lambda_{\varphi}(mp)|\leq |\lambda_{\varphi}(m)\lambda_{\varphi}(p)|+\Big|\lambda_{\varphi}\Big(\frac{m}{p}\Big)\Big|.
\end{align*}
With these inequalities along with the multiplicative structure of the Hecke eigenvalues $\lambda_{\varphi}(m)$, we can mimic the proof of \cite[Theorem 3]{Sound_QUE_Maass} and conclude that
\[
\sum_{m\leq x/y}|\lambda_{\varphi}(m)|^2\ll\frac{\log(ey)}{\sqrt{y}}\sum_{m\leq x}|\lambda_{\varphi}(m)|^2,\qquad 1\leq y\leq x,
\]
hence
\begin{align*}
	\frac{\rho(\varphi)^2}{2}\int_{1}^{\infty}|K_{ir}(2\pi t)|^2\sum_{m\leq t/T}|\lambda_{\varphi}(m)|^2 \, \frac{dt}{t}&\ll \frac{\log T}{\sqrt{T}}\frac{\rho(\varphi)^2}{2}\int_1^{\infty}|K_{ir}(2\pi t)|^2\sum_{m\leq t}|\lambda_{\varphi}(m)|^2 \, \frac{dt}{t}\\
	&=\frac{\log T}{\sqrt{T}}\int_{\substack{|x|\leq \frac{1}{2} \\ y\geq 1}}|\varphi(x+iy)|^2 \, \frac{dx \, dy}{y^2}\\
	&=\frac{\log T}{\sqrt{T}}\mu_{\varphi}(\{z\in\mathbb{C}\colon |\re(z)|\leq \tfrac{1}{2},~\im(z)\geq 1\}).
\end{align*}
Since $\varphi$ is normalized so that $\mu_{\varphi}$ is a probability measure on $\Gamma\backslash\H$ and there exists a fundamental domain of $\Gamma_0(q) \backslash \H$ containing the set $\{z\in\mathbb{C}\colon |\re(z)|\leq \tfrac{1}{2},~\im(z)\geq 1\}$, the preceding display is $\ll (\log T)/\sqrt{T}$.  The result follows.
\end{proof}

We next bound $D_T(\mu_{\varphi})$. Our first step in this regard is to bound this in terms of a spectral expansion on $\Gamma \backslash \mathbb{H}$ in terms of Hecke--Maa\ss{} cusp forms $\varphi_k$ on $\Gamma \backslash \mathbb{H}$ with spectral parameter $t_k$ and Eisenstein series $E(\cdot,\tfrac{1}{2}+it)$.

\begin{lemma}
	\label{lem:discrepancy}
	Let $\varphi\in\mathpzc{F}(Q)$.  If $M,T\geq e$, then
	\[
	D_T(\mu_{\varphi})^2\ll M^{-2} + (1+M^{-3} T)\Big(\sum_{|t_k|\leq M} |\langle \varphi_k,|\varphi|^2 \rangle_{q_{\varphi}}|^2+\int_{|t|\leq M} |\langle E(\cdot,\tfrac{1}{2}+it),|\varphi|^2\rangle_{q_{\varphi}}|^2 \, dt\Big)
	\]
	with an absolute implied constant.
\end{lemma}
\begin{proof}
We follow the strategy in \cite[Section 5]{LS}, which we include for completeness.  Let $B( w ,r)\subseteq \Gamma\backslash\mathbb{H}$ be an injective geodesic ball centered at $ w \in\Gamma\backslash\H$ satisfying $\im( w )<2T$.  Define
\[
k_r(z, w )=\begin{dcases*}
	1& if $d(z, w )<r$,\\
	0& otherwise,
\end{dcases*}
\qquad K_r(z, w ) = \sum_{\gamma\in\Gamma}k_r(\gamma z, w ).
\]
Here
\[d(z,w) \coloneqq \log \frac{|z - \overline{w}| + |z - w|}{|z - \overline{w}| - |z - w|}\]
is the hyperbolic distance between two points $z$ and $w$ in $\H$. It follows from these definitions that $K_r(z, w )=\mathbf{1}_{B( w ,r)}(z)$.  We spectrally expand $K_r(z, w )$ using \cite[Theorem 15.7]{IK}.  If $h_r(t)$ is the Selberg--Harish-Chandra transform of $k_r(z, w )$ (see \cite[Lemma 15.6]{IK}), then
\begin{equation}
\label{eqn:unsmoothed_spectral}
K_r(z, w )=\frac{3}{\pi} h_r\Big(\frac{i}{2}\Big) + \sum_{k=1}^{\infty}h_r(t_k)\varphi_k(z)\overline{\varphi_k( w )}+\frac{1}{4\pi}\int_{\R}h_r(t)E(z,\tfrac{1}{2}+it)\overline{E( w ,\tfrac{1}{2}+it)} \, dt.
\end{equation}

We smooth the sum and the integral in \eqref{eqn:unsmoothed_spectral} as follows.  Let $\psi_{\epsilon}(z, w )$ be a nonnegative mollifier supported inside of a ball of radius $\epsilon$ with the property that $\int_{\H}\psi_{\epsilon}(z, w ) \, d\mu(z) = 1$.  We can and will choose $\psi_{\epsilon}(z, w )$ so that $\psi_{\epsilon}(z, w )\ll \epsilon^{-2}$ and its Selberg--Harish-Chandra transform $h^{(\epsilon)}$ satisfies $|h^{(\epsilon)}(t)|\ll 1$ for $|t|\leq\epsilon^{-1}$ and is rapidly decreasing for $|t|>\epsilon^{-1}$.

Given $B(w,r)$ as above, we consider $B(\zeta,r-2\epsilon)$ and $B(\zeta,r+2\epsilon)$, subject to the convention that if $r\leq 2\epsilon$, then $\mathbf{1}_{B(\zeta,r-2\epsilon)}$ is identically zero.  For a function $F(z)$ on $\Gamma\backslash\mathbb{H}$, we define
\[
\bar{\psi}_{\epsilon}(z,w)=\sum_{\gamma\in\Gamma}\psi_{\epsilon}(\gamma z,w),\qquad (F*\bar{\psi}_{\epsilon})(z)\coloneqq \int_{\Gamma\backslash\mathbb{H}}F(w)\bar{\psi}_{\epsilon}(w,z) \, d\mu(w).
\]
It follows by construction that $k_{r-2\epsilon}*\bar{\psi}_{\epsilon}(z)\leq\mathbf{1}_{B(w,r)}(z)\leq k_{r+2\epsilon}*\bar{\psi}_{\epsilon}(z)$.  These two convolutions have the following expansions per \cite[Equation 48]{LS}:
\begin{multline*}
	k_{r\pm 2\epsilon}*\bar{\psi}_{\epsilon}(z) = \frac{3}{\pi} h_{r\pm2\epsilon}\Big(\frac{i}{2}\Big) h^{(\epsilon)}\Big(\frac{i}{2}\Big) +  \sum_{k=1}^{\infty}h_{r\pm2\epsilon}(t_k)h^{(\epsilon)}(t_k)\varphi_k(z)\overline{\varphi_k(w)}\\
	+\frac{1}{4\pi}\int_{\R}h_{r\pm2\epsilon}(t)h^{(\epsilon)}(t)E(z,\tfrac{1}{2}+it)\overline{E(w,\tfrac{1}{2}+it)} \, dt.
\end{multline*}

For any $H\in L^2(\Gamma\backslash\mathbb{H})$, we by \cite[Theorem 15.5]{IK} and the definition of $\mu_{\varphi}$ that
\[
\mu_{\varphi}(H)=\frac{3}{\pi}\mu(H)+\sum_{k=1}^{\infty}\langle H,\varphi_k\rangle_1 \langle \varphi_k,|\varphi|^2\rangle_{q_{\varphi}}+\frac{1}{4\pi}\int_{-\infty}^{\infty}\langle H,E(\cdot,\tfrac{1}{2}+it)\rangle_1 \langle E(\cdot,\tfrac{1}{2}+it),|\varphi|^2\rangle_{q_{\varphi}} \, dt.
\]
Therefore, we have
\begin{multline*}
\mu_{\varphi}(k_{r\pm 2\epsilon}*\bar{\psi}_{\epsilon})=\frac{3}{\pi} h_{r\pm2\epsilon}\Big(\frac{i}{2}\Big) h^{(\epsilon)}\Big(\frac{i}{2}\Big)+\sum_{k=1}^{\infty}h_{r\pm 2\epsilon}(t_k)h^{(\epsilon)}(t_k)\overline{\varphi_k(w)} \langle \varphi_k,|\varphi|^2\rangle_{q_{\varphi}}\\
+\frac{1}{4\pi}\int_{-\infty}^{\infty}h_{r\pm 2\epsilon}(t)h^{(\epsilon)}(t)\overline{E(w,\tfrac{1}{2}+it)} \langle E(\cdot,\tfrac{1}{2}+it),|\varphi|^2\rangle_{q_{\varphi}} \, dt.
\end{multline*}
The inversion formula for the Selberg--Harish-Chandra transform implies that $h_{r\pm2\epsilon}(\frac{i}{2})$ equals $\int_{\H} \psi_{\epsilon}(z,w) \, d\mu(z) = 1$ while $h_{r\pm2\epsilon}(\frac{i}{2}) = \mu(\mathbf{1}_{B(w,r\pm2\epsilon)}) = \mu(\mathbf{1}_{B(w,r)}) + O(\epsilon)$, so
\begin{multline*}
\Big|\mu_{\varphi}(\mathbf{1}_{B(w,r)})-\frac{3}{\pi}\mu(\mathbf{1}_{B(w,r)})\Big|\ll\epsilon + \sum_{\pm}\Big|\sum_{k=1}^{\infty}h_{r\pm2\epsilon}(t_k)h^{(\epsilon)}(t_k)\overline{\varphi_k(w)}\langle \varphi_k,|\varphi|^2\rangle_{q_{\varphi}}\\
+\int_{-\infty}^{\infty}h_{r\pm 2\epsilon}(t)h^{(\epsilon)}(t)\overline{E(w,\tfrac{1}{2}+it)} \langle E(\cdot,\tfrac{1}{2}+it),|\varphi|^2\rangle_{q_{\varphi}} \, dt\Big|.
\end{multline*}
Consequently,
\begin{multline*}
	D_{T}(\mu_{\varphi})^2\ll\epsilon^2 + \sum_{\pm}\Big(\Big|\sum_{k=1}^{\infty}h_{r\pm2\epsilon}(t_k)h^{(\epsilon)}(t_k)\overline{\varphi_k(w)}\langle \varphi_k,|\varphi|^2\rangle_{q_{\varphi}}\Big|^2\\
+\Big|\int_{-\infty}^{\infty}h_{r\pm 2\epsilon}(t)h^{(\epsilon)}(t)\overline{E(w,\tfrac{1}{2}+it)} \langle E(\cdot,\tfrac{1}{2}+it),|\varphi|^2\rangle_{q_{\varphi}} \, dt\Big|^2\Big).
\end{multline*}

We first handle the contribution from the cuspidal spectrum.  Note that by the Cauchy--Schwarz inequality and our aforementioned decay properties for $h^{(\epsilon)}$, we have
\begin{multline*}
\Big|\sum_{k=1}^{\infty}h_{r\pm2\epsilon}(t_k)h^{(\epsilon)}(t_k)\overline{\varphi_k(w)}\langle \varphi_k,|\varphi|^2 \rangle_{q_{\varphi}}\Big|^2\\
\leq \Big(\sum_{k=1}^{\infty}|h_{r\pm2\epsilon}(t_k)\varphi_k(w)|^2|h^{(\epsilon)}(t_k)|\Big)\sum_{k=1}^{\infty}|\langle \varphi_k,|\varphi|^2 \rangle_{q_{\varphi}}|^2|h^{\epsilon}(t_k)|\\
\ll \Big(\sum_{|t_k|\leq 1/\epsilon}|h_{r\pm2\epsilon}(t_k)\varphi_k(w)|^2\Big)\sum_{|t_k|\leq 1/\epsilon}|\langle \varphi_k,|\varphi|^2 \rangle_{q_{\varphi}}|^2+\epsilon^2.
\end{multline*}
It follows from \eqref{eqn:unsmoothed_spectral} that if $\im(w)\leq 2T$, then
\[
\sum_{|t_k|\leq 1/\epsilon}|h_{r\pm2\epsilon}(t_k)\varphi_k(w)|^2\ll\int_{\Gamma\backslash\mathbb{H}}|K_{r+2\epsilon}(z,w)|^2 \, d\mu(z)\ll 1+\epsilon^3 T,
\]
hence
\[
\Big|\sum_{k=1}^{\infty}h_{r\pm2\epsilon}(t_k)h^{(\epsilon)}(t_k)\overline{\varphi_k(w)}\langle \varphi_k,|\varphi|^2 \rangle_{q_{\varphi}}\Big|^2\ll (1+\epsilon^3 T)\sum_{|t_k|\leq 1/\epsilon}|\langle \varphi_k,|\varphi|^2 \rangle_{q_{\varphi}}|^2.
\]
A verbatim argument for the contribution from the  continuous spectrum shows that
\begin{multline*}
\Big|\int_{-\infty}^{\infty}h_{r\pm 2\epsilon}(t)h^{(\epsilon)}(t)\overline{E(w,\tfrac{1}{2}+it)} \langle E(\cdot,\tfrac{1}{2}+it),|\varphi|^2\rangle_{q_{\varphi}} \, dt\Big|^2\\
\ll (1+\epsilon^3 T)\int_{|t|\leq 1/\epsilon}|\langle E(\cdot,\tfrac{1}{2}+it),|\varphi|^2\rangle_{q_{\varphi}}|^2 \, dt+\epsilon^2.
\end{multline*}
The lemma now follows by replacing $\epsilon$ with $1/M$.
\end{proof}

\subsection{Relating inner products to \texorpdfstring{$L$}{L}-functions}

We now relate the inner products $|\langle \varphi_k,|\varphi|^2 \rangle_{q_{\varphi}}|^2$ and $|\langle E(\cdot,\tfrac{1}{2}+it),|\varphi|^2\rangle_{q_{\varphi}}|^2$ in \cref{lem:discrepancy} to values of $L$-functions on the critical line.  Let $q_{\Ad \varphi}$ be the arithmetic conductor of the adjoint lift $\Ad \varphi$.  The positive integer $q_{\Ad \varphi}$ is a perfect square satisfying $\sqrt{q_{\Ad \varphi}} \mid q_{\varphi}$.  Moreover, we have $\sqrt{q_{\Ad \varphi}}=q_{\varphi}$ if and only if $q_{\varphi}$ is squarefree \cite[Proposition 2.5]{NPS}.

\begin{lemma}
	\label{lem:discrete_spectrum}
	Let $\epsilon'>0$, and let $W_k=W(\varphi_k)\in\{-1,1\}$ be the root number of $\varphi_k$.  We have
	\[
|\langle \varphi_k,|\varphi|^2 \rangle_{q_{\varphi}}|^2\ll_{\epsilon'} (q_{\varphi}(3+|t_{\varphi}|))^{\frac{\epsilon'}{2}} (3+|t_k|)^{\frac{\epsilon'}{2}} \frac{(1+W_k)q_{\Ad \varphi}^{-\frac{1}{2}}\Big(\frac{\sqrt{q_{\Ad \varphi}}}{q_{\varphi}}\Big)^{1-2\vartheta} L(\frac{1}{2},\Ad \varphi \times \varphi_k)}{(3+|t_k|)^{\frac{1}{2}}(3+|2t_{\varphi}-t_k|)^{\frac{1}{2}}(3+|2t_{\varphi}+t_k|)^{\frac{1}{2}}},
\]
where $\vartheta \in [0,\frac{1}{2})$ is the best known exponent towards the generalized Ramanujan conjecture.
\end{lemma}

\begin{remark}
The value $\vartheta = \frac{7}{64}$ is admissible by work of Kim and Sarnak \cite[Appendix 2]{Kim}, so that if $\varphi$ is any Hecke--Maa{\ss} newform and $p$ is any prime, then $|\lambda_{\varphi}(p)| \leq p^{7/64} + p^{-7/64}$.  We also point out that $L(\tfrac{1}{2},\Ad \varphi \times \varphi_k)$ is nonnegative via work of Lapid \cite{Lap03}.
\end{remark}

\begin{proof}
Nelson, Pitale, and Saha \cite[Corollary 2.8, Theorem 3.1, and Proposition 3.3]{NPS} proved that
\[
|\langle \varphi_k,|\varphi|^2 \rangle_{q_{\varphi}}|^2 \leq (1+W_k)\frac{10^{5\omega(q_{\varphi}/\sqrt{q_{\Ad \varphi}})}}{8q_{\varphi}} \frac{\Lambda(\frac{1}{2},\Ad \varphi \times \varphi_k)\Lambda(\frac{1}{2},\varphi_k)}{\Lambda(1,\Ad \varphi )^2 \Lambda(1,\Ad \varphi_k)}\tau\Big(\frac{q_{\varphi}}{\sqrt{q_{\Ad \varphi}}}\Big)^2 \Big(\frac{q_{\varphi}}{\sqrt{q_{\Ad \varphi}}}\Big)^{2\vartheta},
\]
where $\omega(n)$ is the number of prime divisors of $n$ and $\tau(n)$ is the number of divisors of $n$.  While they state their results in the case where $\varphi$ is in fact a holomorphic cuspidal newform of weight $k$, level $q$, and trivial nebentypus, their calculations are purely local.  Therefore, their result carries over to Hecke--Maa\ss{} newforms having trivial nebentypus without any changes.  This uses \cite[Theorem 2]{Watson} to show that the archimedean normalized local integral $I_{\infty}^{\ast}$ in \cite[Theorem 3.1]{NPS} equals $1$ if $\varphi_k$ is even and $0$ if $\varphi_k$ is odd.

In \cite{HL}, it is shown that for all $\epsilon'>0$, we have
\begin{equation}
\label{eqn:HL}
L(1,\Ad \varphi )^{-1}\ll_{\epsilon'} (q_{\varphi}(3+|t_{\varphi}|))^{\frac{\epsilon'}{2}},\qquad L(1,\Ad \varphi_k)^{-1}\ll_{\epsilon'} (3+|t_k|)^{\frac{\epsilon'}{2}}.
\end{equation}
If $\varphi$ is dihedral, then the upper bound on $L(1,\Ad \varphi)^{-1}$ is ineffective because it relies on Siegel's ineffective upper bound on $L(1,\chi)^{-1}$, where $\chi$ is a primitive quadratic Dirichlet character.  The lemma now follows from the definition of $q_{\Ad \varphi}$, Stirling's formula (see \cite{MR3647437} for a similar computation), and the convexity bound for $L(\frac{1}{2},\varphi_k)$.
\end{proof}

\begin{lemma}
\label{lem:continuous_spectrum}
If $t\in\R$ and $\epsilon'>0$, then
\[
|\langle E(\cdot,\tfrac{1}{2}+it),|\varphi|^2\rangle_{q_{\varphi}}|^2
\ll_{\epsilon'} (q_{\varphi}(3+|t_{\varphi}|))^{\frac{\epsilon'}{2}} (3+|t|)^{\frac{\epsilon'}{2}} \frac{q_{\Ad \varphi}^{-\frac{1}{2}}\Big(\frac{\sqrt{q_{\Ad \varphi}}}{q_{\varphi}}\Big)^{1-2\vartheta}|L(\frac{1}{2} + it,\Ad \varphi)|^2}{(3+|t|)^{\frac{1}{2}} (3+|2t_{\varphi}-t|)^{\frac{1}{2}} (3+|2t_{\varphi}+t|)^{\frac{1}{2}}}.
\]
\end{lemma}
\begin{proof}
This is proved using the local calculations of Nelson, Pitale, and Saha in \cite[Corollary 2.8]{NPS} the unfolding method, the convexity bound for the Riemann zeta function, Stirling's formula, and \eqref{eqn:HL}.
\end{proof}

In order to bound the desired averages of the inner products in \cref{lem:discrete_spectrum,lem:continuous_spectrum}, we require an understanding of the central values of the pertinent $L$-functions on average.  We obtain such an understanding using \cref{prop:subconvexity}.  In doing so, we shall identify a Hecke--Maa\ss{} newform $\varphi$ with its corresponding cuspidal automorphic representation $\pi_{\varphi} \in \mathfrak{F}_2$.  Abusing notation, we use $\varphi$ and $\pi_{\varphi}$ interchangeably.  The analytic conductors of $\varphi$ and $\varphi_k$ satisfy $C(\varphi)=q_{\varphi}(3+|t_{\varphi}|)^2$ and $C(\varphi_k)=(3+|t_k|)^2$, respectively.

\begin{proposition}
\label{prop:subconvex_QE}
Let $\epsilon>0$.  Let $1\leq M\leq Q^{1/22}$.  The set
\begin{multline*}
\mathscr{E}_1(Q,M) \coloneqq \{\Ad \varphi\colon \varphi\in\mathpzc{F}(Q),\textup{ there exists $t\in[-M,M]$ such that}\\
|L(\tfrac{1}{2}+it,\Ad \varphi)|\geq (q_{\Ad \varphi}(3+|t_{\varphi}|)^2(3+|t|)^3)^{\frac{1}{4}-6\times 10^{-12}\epsilon}\}
\end{multline*}
has cardinality $O_{\epsilon}(Q^{2\epsilon})$.  Additionally, the set
\begin{multline*}
\mathscr{E}_2(Q,M) \coloneqq \{\Ad \varphi\colon \varphi\in\mathpzc{F}(Q),\textup{ there exists $\varphi_k$ with $|t_k| \leq M$ such that}\\
L(\tfrac{1}{2},\Ad \varphi\times\varphi_k)\geq(q_{\Ad \varphi}^2(3+|t_{\varphi}|)^4(3+|t_k|)^6)^{\frac{1}{4}-6\times 10^{-12}\epsilon}\}
\end{multline*}
has cardinality $O_{\epsilon}(M^2 Q^{2\epsilon})$.
\end{proposition}

\begin{proof}
We give the details for the second part only; the details for the first part are simpler.  We will separately estimate the cardinalities of the nondihedral subfamily
\begin{multline*}
\{\Ad \varphi\colon \varphi\in\mathpzc{F}(Q),\textup{ $\varphi$ nondihedral, there exists $\varphi_k$ with $|t_k| \leq M$ such that}\\
L(\tfrac{1}{2},\Ad \varphi\times\varphi_k)\geq(q_{\Ad \varphi}^2(3+|t_{\varphi}|)^4(3+|t_k|)^6)^{\frac{1}{4}-6\times 10^{-12}\epsilon}\}
\end{multline*}
and the dihedral subfamily
\begin{multline*}
\{\Ad \varphi\colon \varphi\in\mathpzc{F}(Q),\textup{ $\varphi$ dihedral, there exists $\varphi_k$ with $|t_k| \leq M$ such that}\\
L(\tfrac{1}{2},\Ad \varphi\times\varphi_k)\geq(q_{\Ad \varphi}^2(3+|t_{\varphi}|)^4(3+|t_k|)^6)^{\frac{1}{4}-6\times 10^{-12}\epsilon}\}.
\end{multline*}

If $\varphi \in \mathpzc{F}(Q)$ is nondihedral, then it follows from work of Gelbart and Jacquet \cite{GJ} that $\Ad \pi_{\varphi}\in\mathfrak{F}_3$, and $C(\Ad \varphi)\leq 4 C(\varphi)^2$. Therefore, by the above discussion, \cref{prop:subconvexity} implies that the cardinality of the nondihedral subfamily is
\begin{align*}
&\ll \sum_{|t_k|\leq M} \big|\{\pi\in\mathfrak{F}_3(16 Q^2)\colon |L(\tfrac{1}{2},\pi\times \pi_{\varphi_k})|\geq (C(\pi)^2 C(\pi_{\varphi_k})^3)^{\frac{1}{4}-6\times 10^{-12}\epsilon}\}\big|\\
&\ll_{\epsilon}Q^{2\epsilon}|\{\varphi_k\colon |t_k|\leq M\}|.
\end{align*}
This is $O_{\epsilon}(M^2 Q^{2\epsilon})$ by the Weyl law \cite[Chapter 11]{Iwaniec_spectral}
\begin{equation}
\label{eqn:Weyl_law}
|\{t_k\colon |t_k|\leq M\}| = \frac{1}{12}M^2+O(M\log M).
\end{equation}

If $\varphi \in \mathpzc{F}$ is dihedral, then there exists a real quadratic extension $E/\Q$ of discriminant $D > 1$ and a Hecke character $\chi$ of $E$ with arithmetic conductor $\mathfrak{q}$ such that $\varphi$ is the automorphic induction of $\chi$; in particular, $\varphi$ has arithmetic conductor $q_{\varphi} = D \mathrm{N}_{E/\Q}(\mathfrak{q})$ with $\mathrm{N}_{E/\Q}(\mathfrak{q}) \equiv 0 \pmod{D}$ \cite[Lemma 4.2]{Humphries_mult}\footnote{It is erroneously stated in \cite[Lemma 4.2]{Humphries_mult} that $\mathrm{N}_{E/Q}(\mathfrak{q}) = D$. Daniel Hu alerted the first author that in fact only the weaker statement $\mathrm{N}_{E/\Q}(\mathfrak{q}) \equiv 0 \pmod{D}$ is true.}. The adjoint lift of $\varphi$ has the isobaric decomposition $\Ad \varphi=\chi_D\boxplus\varphi'\otimes\chi_D$, where $\varphi'$ is the automorphic induction of $\chi^2$ and $\chi_D$ denotes the primitive Dirichlet character modulo $D$ corresponding to $E/\Q$.  This gives us the factorization
\[
L(\tfrac{1}{2},\Ad \varphi\times\varphi_k)=L(\tfrac{1}{2},\varphi_k\otimes\chi_D)L(\tfrac{1}{2},(\varphi'\otimes\chi_D)\times\varphi_k),
\]
and both central $L$-values on the right-hand side are nonnegative \cite{Wal85}. It follows that if $\varphi \in \mathpzc{F}(Q)$ is dihedral, the discriminant of the associated real quadratic field $E$ satisfies $D \leq \sqrt{2Q}$, while the analytic conductor of $C(\varphi_k\otimes\chi_D)$ is $D^2 C(\varphi_k)$, and
\[
C((\varphi'\otimes\chi_D)\times\varphi_k)=q_{\varphi'\otimes\chi_D}^2(3+|2t_{\varphi}+t_k|)^2(3+|2t_{\varphi}-t_k|)^2,
\]
where $D q_{\varphi'\otimes\chi_D}$ is a perfect square for which $\sqrt{D q_{\varphi'\otimes \chi_D}} \mid q_{\varphi}$, so that $q_{\varphi' \otimes \chi_D} \leq 4Q^2$. 
So the cardinality of the dihedral subfamily is
\begin{multline*}
\ll \sum_{|t_k|\leq M} \big|\{D\leq\sqrt{2Q}\colon L(\tfrac{1}{2},\varphi_k\otimes\chi_D) \geq (D^2 C(\varphi_k))^{\frac{1}{4}-6\times 10^{-12}\epsilon}\}\big|\\
+ \sum_{|t_k|\leq M} \big|\{\pi\in\mathfrak{F}_2(4Q^2)\colon |L(\tfrac{1}{2},\pi\otimes\pi_{\varphi_k})|\geq (C(\pi)^2 C(\varphi_k)^2)^{\frac{1}{4}-6\times 10^{-12}\epsilon}\}\big|.
\end{multline*}
By \cref{prop:subconvexity} and \eqref{eqn:Weyl_law}, this is $O_{\epsilon}(M^2 Q^{2\epsilon})$.  The proposition follows by combining the dihedral and nondihedral subfamilies.
\end{proof}

\begin{corollary}
\label{cor:subconvexity}
Let $\epsilon,\epsilon'>0$.  Let $1\leq M\leq Q^{1/22}$.  If $\Ad\varphi\notin \mathscr{E}_1(Q,M)$, then
\[
\int_{|t|\leq M} |\langle E(\cdot,\tfrac{1}{2}+it),|\varphi|^2\rangle_{q_{\varphi}}|^2 \, dt\ll_{\epsilon'}(\lambda_{\varphi}q_{\varphi})^{\epsilon'-1.2\times 10^{-11}\epsilon} M^{\frac{7}{2}}.
\]
Also, if $\Ad\varphi\notin\mathscr{E}_2(Q,M)$, then
\[
\sum_{|t_k|\leq M} |\langle \varphi_k,|\varphi|^2 \rangle_{q_{\varphi}}|^2\ll_{\epsilon'}(\lambda_{\varphi}q_{\varphi})^{\epsilon'-1.2\times 10^{-11}\epsilon}M^{\frac{7}{2}}.
\]
\end{corollary}
\begin{proof}
We give the details for the second part only; the details for the first part are simpler.  Note that if $\varphi\in\mathpzc{F}(Q)$, then $\lambda_{\varphi}q_{\varphi}\asymp Q$.  By \cref{lem:discrete_spectrum}, we have that
\begin{equation}
\label{eqn:new_sum_1}
\sum_{|t_k|\leq M} |\langle \varphi_k,|\varphi|^2 \rangle_{q_{\varphi}}|^2\ll_{\epsilon'} Q^{\epsilon'} \sum_{|t_k|\leq M}\frac{q_{\Ad \varphi}^{-\frac{1}{2}}(\sqrt{q_{\Ad \varphi}}/q_{\varphi})^{1-2\vartheta}L(\frac{1}{2},\Ad \varphi \times \varphi_k)}{[(3+|t_k|)(3+|2t_{\varphi}-t_k|)(3+|2t_{\varphi}+t_k|)]^{1/2}}
\end{equation}
\cref{prop:subconvex_QE} and \cite[Example 3]{ST} ensure that for all $k$ such that $|t_k|\leq M$, we have the bound
\[
L(\tfrac{1}{2},\Ad \varphi \times \varphi_k)\ll \begin{dcases*}
[q_{\Ad \varphi}^{2}(3+|t_{\varphi}|)^4 (3+|t_k|)^6]^{\frac{1}{4}}& if $\Ad \varphi\in\mathscr{E}_1(Q,M)$,\\
[q_{\Ad \varphi}^{2}(3+|t_{\varphi}|)^4 (3+|t_k|)^6]^{\frac{1}{4}-6\times 10^{-12}\epsilon}& if $\Ad \varphi\notin\mathscr{E}_1(Q,M)$.
\end{dcases*}
\]
Therefore, if $\Ad \varphi\notin\mathscr{E}_2(Q,M)$, then \eqref{eqn:new_sum_1} is
\begin{equation}
\label{eqn:newsum2}
\begin{aligned}
\ll_{\epsilon'} Q^{\epsilon'} \sum_{|t_k|\leq M} \frac{(\sqrt{q_{\Ad \varphi}}/q_{\varphi})^{1-2\vartheta}}{(q_{\Ad \varphi}^{2}(3+|t_{\varphi}|)^4)^{6\epsilon/10^{12}}}\frac{(3+|t_k|)(3+|t_{\varphi}|)}{(3+|2t_{\varphi}-t_k|)^{\frac{1}{2}} (3+|2t_{\varphi}+t_k|)^{\frac{1}{2}}}.
\end{aligned}
\end{equation}

Since $\sqrt{q_{\Ad \varphi}} \mid q_{\varphi}$ (with equality if and only if $q_{\varphi}$ is squarefree), we have
\[
\frac{(\sqrt{q_{\Ad \varphi}}/q_{\varphi})^{1-2\vartheta}}{(q_{\Ad \varphi}^2(3+|t_{\varphi}|)^{4})^{6\epsilon/10^{12}}}\ll Q^{-1.2\times 10^{-11}\epsilon}.
\]
The bound
\[
\frac{(3+|t_k|)(3+|t_{\varphi}|)}{(3+|2t_{\varphi}-t_k|)^{\frac{1}{2}} (3+|2t_{\varphi}+t_k|)^{\frac{1}{2}}} \ll 1 + |t_k|^{\frac{3}{2}}\ll M^{\frac{3}{2}}
\]
holds since the supremum of the left-hand side as $t_{\varphi}$ varies is achieved when $2t_{\varphi}=\pm t_k$.  Therefore, by the above discussion and \eqref{eqn:Weyl_law}, \eqref{eqn:newsum2} is
\[
\ll_{\epsilon'}Q^{\epsilon'-1.2\times 10^{-11}\epsilon} M^{\frac{7}{2}}\asymp_{\epsilon'}(\lambda_{\varphi}q_{\varphi})^{\epsilon'-1.2\times 10^{-11}\epsilon} M^{\frac{7}{2}}.\qedhere
\]
\end{proof}

\subsection{Proofs of \texorpdfstring{\cref{thm:QE}}{Theorem \ref{thm:QE}} and \texorpdfstring{\cref{cor:QE}}{Corollary \ref{cor:QE}}}

\begin{proof}[{Proof of \cref{thm:QE}}]
Let
\[
T\geq e,\qquad 1\leq M\leq Q^{\frac{1}{22}},\qquad \epsilon,\epsilon'>0.
\]
By \cref{lem:Sound,lem:discrepancy}, we find that if $\varphi \in \mathpzc{F}(Q)$, then $D(\mu_{\varphi})^2$ is
\begin{align*}
\ll\frac{(\log T)^2}{T}+\frac{1}{M^2}+\Big(1+\frac{T}{M^3}\Big)\Big(\sum_{|t_k|\leq M} |\langle \varphi_k,|\varphi|^2 \rangle_{q_{\varphi}}|^2+\int_{|t|\leq M} |\langle E(\cdot,\tfrac{1}{2}+it),|\varphi|^2\rangle_{q_{\varphi}}|^2 \, dt\Big).
\end{align*}
Recall $\mathscr{E}_1(Q,M)$ and $\mathscr{E}_2(Q,M)$ from \cref{prop:subconvex_QE}.  If $\Ad\varphi\notin\mathscr{E}_1(Q,M)\cup\mathscr{E}_2(Q,M)$, then by \cref{cor:subconvexity}, we have the bound
\[
D(\mu_{\varphi})^2\ll_{\epsilon'} \frac{(\log T)^2}{T}+\frac{1}{M^{2}} + \Big(1+\frac{T}{M^3}\Big)(\lambda_{\varphi}q_{\varphi})^{\epsilon'-1.2\times 10^{-11}\epsilon} M^{\frac{7}{2}}.
\]
We choose $Q$ to be large with respect to $\epsilon$, and we choose
\[
A=\frac{6875\cdot 10^8}{3},\qquad T=Q^{\frac{\epsilon}{A}},\qquad M = Q^{\frac{\epsilon}{2A}},\qquad \epsilon'=\frac{\epsilon}{10^{20}}.
\]
Noting that $|\mathscr{E}_1(Q,M)| + |\mathscr{E}_2(Q,M)| \ll_{\epsilon} M^2 Q^{2\epsilon}$ by \cref{prop:subconvex_QE}, we find that
\[
|\{\Ad \varphi\colon\varphi\in\mathpzc{F}(Q),~D(\mu_{\varphi})\geq (\lambda_{\varphi}q_{\varphi})^{-2.18\times 10^{-12}\epsilon}\}|\ll_{\epsilon}Q^{(2+\frac{1}{A})\epsilon}.
\]
Rescaling $\epsilon$ to $\epsilon/(2+1/A)$, we conclude that
\[
|\{\Ad \varphi\colon\varphi\in\mathpzc{F}(Q),~D(\mu_{\varphi})\geq (\lambda_{\varphi}q_{\varphi})^{-1.08\times 10^{-12}\epsilon}\}|\ll_{\epsilon}Q^{\epsilon},
\]
which is stronger than what \cref{thm:QE} asserts.
\end{proof}

\begin{proof}[{Proof of \cref{cor:QE}}]
Given $\varphi\in\mathpzc{F}(Q)$, define $m(\varphi,Q)\coloneqq \{\varphi'\in\mathpzc{F}(Q)\colon \Ad\varphi=\Ad\varphi'\}$.  By \cite[Theorem 4.1.2]{Ramakrishnan}, if $q_{\varphi}$ is squarefree, then $m(\varphi,Q)=1$.  Otherwise, for all $\delta>0$, we have the bound $m(\varphi,Q)\ll_{\delta}Q^{1/2+\delta}$. In light of the bound \eqref{eqn:F(Q)_bound} and the fact that the convexity bound for $L$-functions yields $|D(\mu_{\varphi})|\ll 1$, \cref{cor:QE} follows immediately from \cref{thm:QE}.
\end{proof}

\section{Nonsplit quantum ergodicity}
\label{sec:nonsplit}

Let $E = \Q(\sqrt{D})$ be a real quadratic field with ring of integers $\cO_E$, where $D > 0$ is a fundamental discriminant; we assume for simplicity that $E$ has narrow class number $1$.  Let $\sigma$ be the nontrivial Galois automorphism of $E$.  Given
\[
\gamma = \begin{pmatrix} a & b \\ c & d \end{pmatrix}\in \SL_2(\cO_E),
\]
define
\[
\gamma z \coloneqq \frac{az + b}{cz + d},\qquad \sigma(\gamma) z \coloneqq \frac{\sigma(a)z + \sigma(b)}{\sigma(c)z + \sigma(d)}.
\]
A Hilbert Hecke--Maa\ss{} cusp form of level $\cO_E$ is an $L^2$-normalized smooth function $\phi \colon \H \times \H \to \mathbb{C}$ for which
\begin{itemize}
\item $\phi$ is a joint eigenfunction of the weight $0$ Laplacians
\[
\Delta_1 \coloneqq -y_1^2 \Big(\frac{\partial^2}{\partial x_1^2} + \frac{\partial^2}{\partial y_1^2}\Big), \qquad \Delta_2 \coloneqq -y_2^2 \Big(\frac{\partial^2}{\partial x_2^2} + \frac{\partial^2}{\partial y_2^2}\Big)
\]
for $(z_1,z_2) = (x_1 + iy_1,x_2 + iy_2) \in \H \times \H$, so that there exist $t_{1,\phi},t_{2,\phi} \in \R \cup i[-\frac{7}{64},\frac{7}{64}]$ such that if $\lambda_{1,\phi} = \frac{1}{4} + t_{1,\phi}^2$ and $\lambda_{2,\phi} = \frac{1}{4} + t_{2,\phi}^2$, then
\[
\Delta_1 \phi(z_1,z_2) = \lambda_{1,\phi} \phi(z_1,z_2),\qquad \Delta_2 \phi(z_1,z_2) = \lambda_{2,\phi} \phi(z_1,z_2),
\]
\item $\phi$ is automorphic, so that if $\gamma\in \mathrm{SL}_2(\mathcal{O}_E)$, then $\phi(\gamma z_1, \sigma(\gamma) z_2) = \phi(z_1,z_2)$,
\item $\phi$ is of moderate growth,
\item $\phi$ is cuspidal, and
\item $\phi$ is a joint eigenfunction of every Hecke operator.
\end{itemize}
There is a diagonal embedding $\H \hookrightarrow \H \times \H$ given by the map $z \mapsto (z,z)$. A Hilbert Hecke--Maa\ss{} cusp form $\phi$ is $\Gamma$-invariant when restricted to the diagonal embedding of $\H$; thus $\phi(z,z)$ may be viewed as the restriction of a Hilbert Hecke--Maa\ss{} cusp form to the modular surface $\Gamma \backslash \H$.

\begin{remark}
The constraints $t_{1,\phi},t_{2,\phi} \in \R \cup i[-\frac{7}{64},\frac{7}{64}]$ follow from progress of Blomer and     Brumley towards the generalized Ramanujan conjecture \cite[Theorem 1]{BB2}. We only require the weaker fact that $t_{1,\phi},t_{2,\phi} \in \R \cup i(-\frac{1}{6},\frac{1}{6})$, which is needed in the proof of \cref{lem:realtriple} below in order to invoke work of Cheng \cite{Che21}.
\end{remark}

\subsection{Period integrals involving Hilbert Maa\ss{} cusp forms}

We consider $\phi(z,z)$ integrated over $\Gamma \backslash \H$ against a Laplacian eigenfunction $H$. By assumption, $\phi$ is $L^2$-normalized, so that
\[
\int\limits_{\SL_2(\cO_E) \backslash \H \times \H} |\phi(z_1,z_2)|^2 \, d\mu(z_1,z_2) = 1,\qquad d\mu(z_1,z_2) \coloneqq \frac{dx_1 \, dx_2 \, dy_1 \, dy_2}{y_1^{2} y_2^{2}}.
\]
We have that $\vol(\SL_2(\cO_E) \backslash \H \times \H) = 2 \sqrt{D} \xi_E(2)$, where $\xi_E(s) \coloneqq D^{s/2} \Gamma_{\R}(s)^2 \zeta_E(s)$ denotes the completed Dedekind zeta function. There are three cases of interest:

\begin{enumerate}
\item $H$ is a constant,
\item $H$ is an Eisenstein series,
\item $H$ is a Hecke--Maa\ss{} cusp form.
\end{enumerate}
In each case, the corresponding period integral $\int_{\Gamma \backslash \H} \phi(z,z) H(z) \, d\mu(z)$ may be associated to certain $L$-functions, as we now elucidate; we postpone the proofs of these identities to \cref{sect:proofsperiods}.

\subsubsection{Nonsplit quantum limits}

We first consider the case of $H = 1$. We completely classify the possible values of $\int_{\Gamma \backslash \H} \phi(z,z) \, d\mu(z)$. We may heuristically think of these possible values as specifying the off-diagonal quantum limits in this nonsplit setting, since \cref{conj:realQUE} predicts that the difference
\[\int_{\Gamma \backslash \H} H(z) \phi_j(z,z) \, d\mu(z) - \frac{3}{\pi} \int_{\Gamma \backslash \H} \phi_j(z,z) \, d\mu(z) \int_{\Gamma \backslash \H} H(z) \, d\mu(z)\]
converges to $0$. Note, however, that $\int_{\Gamma \backslash \H} \phi(z,z) \, d\mu(z)$ may fluctuate based on the Hilbert Hecke--Maa\ss{} cusp form $\phi$, so that these are not true quantum limits as they are not independent of $\phi$.

\begin{lemma}
\label{lem:realBS}
Fix a real quadratic number field $E =\Q(\sqrt{D})$ with narrow class number $1$, and denote by $\chi_D$ the quadratic Dirichlet character modulo $D$ associated to $E$. Let $\phi$ be a Hilbert Hecke--Maa\ss{} cusp form with positive first Fourier coefficient. Then
\begin{multline}
\label{eqn:realBS}
\int_{\Gamma \backslash \H} \phi(z,z) \, d\mu(z)	\\
= \begin{dcases*}
\frac{\sqrt{2}}{D^{\frac{1}{4}}} \sqrt{\frac{\Lambda(1, \Ad \varphi \otimes \chi_D)}{\Lambda(1, \Ad \varphi)}} & \parbox{.62\textwidth}{if $\phi$ is the base change of a nondihedral Hecke--Maa\ss{} newform $\varphi$ of weight $0$, level $D$, nebentypus $\chi_D$, and Laplacian eigenvalue $\lambda_{\varphi} = \lambda_{1,\phi} = \lambda_{2,\phi}$,}	\vspace{1mm}\\
0 & otherwise.
\end{dcases*}
\end{multline}
\end{lemma}

\begin{remark}
From this, one can readily show that there exist absolute constants $c_1,c_2 > 0$ such that if $\phi$ is indeed the base change of $\varphi$, then
\begin{equation}
\label{eqn:realBSbounds}
\exp(-c_1 \sqrt{\log C(\As\phi)}) \ll_{D} \int_{\Gamma \backslash \H} \phi(z,z) \, d\mu(z) \ll_{D} \exp(c_2 \sqrt{\log C(\As\phi)}).
\end{equation}
These are consequences of \cite{Banks,HL,Li}.
\end{remark}

\subsubsection{Restrictions of Hilbert Hecke--Maa\ss{} cusp forms and Eisenstein series}

Next, we take $H$ to be an Eisenstein series $E(z,\frac{1}{2}+ it)$ with $t \in \R$.

\begin{lemma}[{Cf.~\cite[Lemma 4.3]{Che21}}]
\label{lem:realEis}
Let $\phi$ be a Hilbert Hecke--Maa\ss{} cusp form with positive first Fourier coefficient, and suppose that $t \in \R$. Then
\begin{equation}
\label{eqn:realEis}
\int_{\Gamma \backslash \H} \phi(z,z) E(z,\tfrac{1}{2} + it) \, d\mu(z) = \frac{1}{\sqrt{2} D^{\frac{1}{4}}} \frac{\Lambda(\frac{1}{2} + it, \As \phi)}{\sqrt{\Lambda(1, \Ad \phi)} \xi(1 + 2it)}.
\end{equation}
\end{lemma}

Here $\As \phi$ denotes the Asai transfer of $\phi$, as introduced in \cite{Asa77}, while $\xi(s) \coloneqq \Gamma_{\R}(s) \zeta(s)$ denotes the completed Riemann zeta function.

\subsubsection{Restrictions of Hilbert Hecke--Maa\ss{} cusp forms and Hecke--Maa\ss{} cusp forms}

Finally, we take $H$ to be a Hecke--Maa\ss{} cusp form $\varphi_j$.

\begin{lemma}[{Cf.~\cite[Theorem 5.6]{Che21}}]
\label{lem:realtriple}
Let $\phi$ be a Hilbert Hecke--Maa\ss{} cusp form and let $\varphi_k$ be a Hecke--Maa\ss{} cusp form on $\Gamma \backslash \H$ of parity $W_k \in \{1,-1\}$. Then
\[\Big|\int_{\Gamma \backslash \H} \phi(z,z) \varphi_k(z) \, d\mu(z)\Big|^2 = \frac{1 + W_k}{8\sqrt{D}} \frac{\Lambda(\frac{1}{2}, \As \phi \times \varphi_k)}{\Lambda(1, \Ad \phi) \Lambda(1, \Ad \varphi_k)}.\]
\end{lemma}

\begin{remark}
The central $L$-value $L(\tfrac{1}{2},\As \phi \times \varphi_k)$ is nonnegative \cite{Lap03}.
\end{remark}

\subsubsection{Conditional bounds}

For the sake of posterity, we record bounds towards these period integrals under the assumption of the generalized Lindel\"{o}f hypothesis.

\begin{lemma}
Assume the generalized Lindel\"{o}f hypothesis. Let $\phi$ be a Hilbert Hecke--Maa\ss{} cusp form and let $t \in \R$. Then
\[\int_{\Gamma \backslash \H} \phi(z,z) E(z,\tfrac{1}{2} + it) \, d\mu(z) \ll_{D,t,\epsilon} C(\As \phi)^{-\frac{1}{4} + \epsilon}.\]
Similarly, let $\varphi_k$ be a Hecke--Maa\ss{} cusp form on $\Gamma \backslash \H$. Then
\[\int_{\Gamma \backslash \H} \phi(z,z) \varphi_k(z) \, d\mu(z) \ll_{D,t_k,\epsilon} C(\As \phi)^{-\frac{1}{4} + \epsilon}.\]
\end{lemma}

\begin{proof}
We prove the latter; the former follows similarly. Via \cref{lem:realtriple}, it suffices to show that
\begin{equation}
\label{eqn:Lindeloftoprove}
\frac{\Lambda(\frac{1}{2}, \As \phi \times \varphi_k)}{\Lambda(1, \Ad \phi) \Lambda(1, \Ad \varphi_k)} \ll_{D,t_k,\epsilon} C(\As \phi)^{-\frac{1}{2} + \epsilon},
\end{equation}
and we may assume without loss of generality that $W_k = 1$. Using the generalized Lindel\"{o}f hypothesis (for the numerator) and \eqref{eqn:HL} (for the denominator), we find that
\[
\frac{\Lambda(\frac{1}{2}, \As \phi \times \varphi_k)}{\Lambda(1, \Ad \phi) \Lambda(1, \Ad \varphi_k)}\ll_{D,t,k,\epsilon}C(\Ad\phi)^{\epsilon}\frac{\prod_{\pm_1, \pm_2, \pm_3} \Gamma_{\R}(\frac{1}{2} \pm_1 it_{1,\phi} \pm_2 it_{2,\phi} \pm_3 it_k)}{\Gamma_{\R}(1)^2 \prod_{\pm} \Gamma_{\R}(1 \pm 2it_{1,\phi}) \Gamma_{\R}(1 \pm 2it_{2,\phi}) \Gamma_{\R}(1 \pm 2it_k)}.
\]
By Stirling's formula, we have the asymptotic formula
\[
|\Gamma_{\R}(\sigma + i\tau)| = 2^{1 - \frac{\sigma}{2}} \pi^{\frac{1 - \sigma}{2}} (3 + |\tau|)^{\frac{\sigma - 1}{2}} e^{-\frac{\pi}{4}|\tau|} \left(1 + O_{\sigma}\left(\frac{1}{3 + |\tau|}\right)\right).
\]
This ratio of gamma functions is therefore equal to
\[
8\pi e^{-\pi \Omega(t_k,t_{1,\phi},t_{2,\phi})} \prod_{\pm_1, \pm_2} (3 + |t_{1,\phi} \pm_1 t_{2,\phi} \pm_2 t_k|)^{-\frac{1}{2}}(1 + \mathcal{E}(t_{1,\phi},t_{2,\phi},t_k)),
\]
where
\[
\mathcal{E}(t_{1,\phi},t_{2,\phi},t_k)\ll \frac{1}{3 + |t_{1,\phi}|} + \frac{1}{3 + |t_{2,\phi}|} + \frac{1}{3 + |t_k|} + \sum_{\pm_1,\pm_2} \frac{1}{3 + |t_{1,\phi} \pm_1 t_{2,\phi} \pm_2 t_k|}
\]
and
\begin{equation}
\label{eqn:Omegadef}
\Omega(t,t_{1,\phi},t_{2,\phi})
\coloneqq \begin{dcases*}
0 & \parbox{.4\textwidth}{if $|t_{1,\phi}| \geq |t_{2,\phi}|$ and $|t_{1,\phi}| - |t_{2,\phi}| \leq |t| \leq |t_{1,\phi}| + |t_{2,\phi}|$ or $|t_{2,\phi}| \geq |t_{1,\phi}|$ and $|t_{2,\phi}| - |t_{1,\phi}| \leq |t| \leq |t_{1,\phi}| + |t_{2,\phi}|$,}	\\
|t_{1,\phi}| - |t_{2,\phi}| - |t| & if $|t_{1,\phi}| \geq |t_{2,\phi}|$ and $|t| \leq |t_{1,\phi}| - |t_{2,\phi}|$,	\\
|t_{2,\phi}| - |t_{1,\phi}| - |t| & if $|t_{2,\phi}| \geq |t_{1,\phi}|$ and $|t| \leq |t_{2,\phi}| - |t_{1,\phi}|$,	\\
|t| - |t_{1,\phi}| - |t_{2,\phi}| & if $|t| \geq |t_{1,\phi}| + |t_{2,\phi}|$.
\end{dcases*}\hspace{-3mm}
\end{equation}
The result then follows.
\end{proof}

\subsubsection{Unconditional bounds}

As an application of \cref{prop:subconvexity}, we are able to unconditionally prove subconvex bounds towards these period integrals provided one excises a sparse subfamily of Hilbert Hecke--Maa\ss{} cusp forms. In order to show this, we first require some auxiliary results on fibers of the Asai transfer.

\begin{lemma}
\label{lem:Asaifibres}
Let $E$ be a real quadratic extension of $\Q$ with narrow class number $1$, and let $\Pi$ and $\Pi'$ be cuspidal automorphic representations of $\GL_2(\A_E)$ of arithmetic conductor $\cO_E$. The Asai transfers $\As \Pi$ and $\As \Pi'$ are equal if and only if $\Pi' \in \{\Pi,\Pi^{\sigma}\}$, where $\sigma$ generates $\mathrm{Gal}(E/\Q)$.
\end{lemma}

\begin{proof}
From \cite[Theorem 7.1]{Kri03}, if $\Pi$ and $\Pi'$ are cuspidal automorphic representations of $\GL_2(\A_E)$ for which $\As \Pi = \As \Pi'$, then there exists a Hecke character $\omega$ of $E^{\times} \backslash \A_E^{\times}$ such that either $\Pi = \Pi' \otimes \omega$ or $\Pi^{\sigma} = \Pi' \otimes \omega$. Since $\Pi$ and $\Pi'$ are of arithmetic conductor $\cO_E$, and hence are unramified at every nonarchimedean place, $\omega$ must also be unramified at every nonarchimedean place. The number of such characters is the narrow class number of $E$.  Since the narrow class number equals $1$ by hypothesis, $\omega$ must be the trivial character.
\end{proof}

\begin{lemma}
\label{lem:BC}
Let $E$ be a real quadratic extension of $\Q$ with narrow class number $1$, and let $\Pi$ be a cuspidal automorphic representation of $\GL_2(\A_E)$ of arithmetic conductor $\cO_E$. If $\As \Pi$ is noncuspidal, then $\Pi$ is the base change of a nondihedral cuspidal automorphic representation $\pi$ of $\GL_2(\A_{\Q})$ of arithmetic conductor $D$, and central character $\omega_{E/\Q}$, the quadratic Hecke character of $\Q^{\times} \backslash \A_{\Q}^{\times}$ corresponding to the quadratic extension $E/\Q$. We have the isobaric decomposition $\As \Pi = (\Ad \pi \otimes \omega_{E/\Q}) \boxplus \mathbbm{1}$, and $\pi$ is unique up to a twist by $\omega_{E/\Q}$. Finally, if $\Pi$ and $\Pi'$ are the base changes of $\pi$ and $\pi'$ respectively, then $\Ad \pi = \Ad \pi'$ if and only if $\Pi = \Pi'$.
\end{lemma}

\begin{proof}
Necessarily, $\Pi$ must be nondihedral, since it is unramified at every nonarchimedean place, so \cite[Theorem B (a)]{Kri12} implies that $\As \Pi$ is noncuspidal if and only if $\Pi = \Pi^{\sigma}$. From the work of Langlands \cite{Lan80}, the condition $\Pi = \Pi^{\sigma}$ can only be met if $\Pi$ is the base change of a cuspidal automorphic representation $\pi$ of $\GL_2(\A_{\Q})$. In this case, the automorphic induction of $\Pi$ to an automorphic representation $\mathcal{AI}_{E/\Q} \Pi$ of $\GL_4(\A_{\Q})$ is noncuspidal and has the isobaric decomposition $\pi \boxplus (\pi \otimes \omega_{E/\Q})$. By comparing these representations, we see that the central character of $\pi$ must be $\omega_{E/\Q}$ and the arithmetic conductor of $\pi$ must be $D$.

The cuspidal automorphic representation $\pi$ must be unique up to a twist by $\omega_{E/\Q}$, since it is shown in \cite{Lan80} that two cuspidal automorphic representations $\pi$ and $\pi'$ of $\GL_2(\A_{\Q})$ have identical base change if and only if $\pi' = \pi \otimes \omega_{E/\Q}$. Furthermore, Krishnamurthy \cite[Theorem B (c)]{Kri12} has established the isobaric decomposition $\As \Pi = (\Ad \pi \otimes \omega_{E/\Q}) \boxplus 1$.

Next, we observe that $\pi$ must be nondihedral, for otherwise there would exist some Hecke character $\chi$ of $E^{\times} \backslash \A_{E}^{\times}$ such that $\pi$ is the automorphic induction $\mathcal{AI}_{E/\Q} \chi$ of $\chi$, but then $\Pi$ would have the isobaric decomposition $\chi \boxplus \chi$, and in particular would not be cuspidal.

Finally, we note that if $\Ad \pi = \Ad \pi'$, then from \cite[Theorem 4.1.2]{Ramakrishnan}, there must exist a Hecke character $\omega$ of $\Q^{\times} \backslash \A_E^{\times}$ such that $\pi = \pi' \otimes \omega$. Since $\pi$ and $\pi'$ both have arithmetic conductor $D$ and central character $\omega_{E/\Q}$, necessarily $\omega$ is either trivial or equal to $\omega_{E/\Q}$; in either case, the base change of $\pi'$ must be equal to that of $\pi$.
\end{proof}

We may use \cref{lem:BC} to give a lower bound for $\mathpzc{F}_{\As}(Q) \coloneqq \{\phi \colon C(\As \phi)\in[Q,2Q]\}$.
\begin{lemma}
\label{lem:Asaisize}
If $Q \geq 1$, then $|\mathpzc{F}_{\As}(Q)| \gg_D \sqrt{Q}$.
\end{lemma}

\begin{proof}
By positivity, it suffices to bound from below the number of $\phi \in \mathpzc{F}_{\As}(Q)$ that are a base change, as in \cref{lem:BC}. This is precisely the number of nondihedral Hecke--Maa\ss{} newforms $\varphi$ of weight $0$, level $D$, nebentypus $\chi_D$, and spectral parameter $t_{\varphi} \in [\tfrac{\sqrt{Q}}{3} - 3, \tfrac{\sqrt{2Q}}{3} - 3]$, which is $\gg_D \sqrt{Q}$ by \eqref{eqn:Weyl_law}.
\end{proof}

We now apply \cref{prop:subconvexity} to prove subconvex bounds for almost all period integrals.

\begin{proposition}
\label{prop:Asaibounds}
Let $\epsilon>0$. Let $1 \leq M \leq Q^{1/22}$. The set
\begin{multline*}
\mathscr{D}_1(Q,M) \coloneqq \big\{\phi \in \mathpzc{F}_{\As}(Q): \textup{there exists $t \in [-M,M]$ such that }	\\
L(\tfrac{1}{2} + it, \As \phi) \geq C(\As \phi)^{\frac{1}{4} - \frac{\epsilon}{18 \cdot 10^{11}}} (1 + |t|)\big\}
\end{multline*}
has cardinality $O_{D,\epsilon}(Q^{\epsilon})$. Additionally, the set
\begin{multline*}
\mathscr{D}_2(Q,M) \coloneqq \big\{\phi \in \mathpzc{F}_{\As}(Q): \textup{there exists $\varphi_k$ with $|t_k| \leq M$ such that}	\\
L(\tfrac{1}{2}, \As \phi \times \varphi_k) \geq C(\As \phi)^{\frac{1}{2} - \frac{\epsilon}{9 \cdot 10^{11}}} C(\varphi_k)\big\}
\end{multline*}
has cardinality $O_{D,\epsilon}(M^2 Q^{\epsilon})$.
\end{proposition}

\begin{proof}
We give the details for the second part only; the details for the first part are simpler. Given a Hilbert Hecke--Maa\ss{} cusp form $\phi \in \mathpzc{F}_{\As}(Q)$ with archimedean spectral parameters $t_{1,\phi}$ and $t_{2,\phi}$, let $\Pi$ denote the underlying cuspidal automorphic representation of $\GL_2(\A_E)$.  This has arithmetic conductor $\cO_E$, and its two archimedean components are principal series representations with spectral parameters $t_{1,\phi}$ and $t_{2,\phi}$. The Asai transfer $\As \Pi$ of $\Pi$ is an automorphic representation of $\GL_4(\A_{\Q})$ of analytic conductor $C(\As \phi) D$. If $\Pi$ is the base change of $\pi$, as in \cref{lem:BC}, then as $\pi$ is nondihedral, $\Ad \pi \otimes \omega_{E/\Q}$ is a cuspidal automorphic representation of $\GL_3(\A_{\Q})$ and also has analytic conductor $C(\As \phi) D$.

We now bound the cardinality of $\mathscr{D}_2(Q,M)$ by separately estimating the cardinalities of the cuspidal Asai transfer subfamily
\begin{multline*}
\big\{\phi \in \mathpzc{F}_{\As}(Q): \As \phi \textup{ cuspidal, there exists $\varphi_k$ with $|t_k| \leq M$ such that}	\\
L(\tfrac{1}{2}, \As \phi \times \varphi_k) \geq C(\As \phi)^{\frac{1}{2} - \frac{\epsilon}{9 \cdot 10^{11}}} C(\varphi_k)\big\}
\end{multline*}
and the noncuspidal Asai transfer subfamily
\begin{multline*}
\big\{\phi \in \mathpzc{F}_{\As}(Q): \As \phi \textup{ noncuspidal, there exists $\varphi_k$ with $|t_k| \leq M$ such that}	\\
L(\tfrac{1}{2}, \As \phi \times \varphi_k) \geq C(\As \phi)^{\frac{1}{2} - \frac{\epsilon}{9 \cdot 10^{11}}} C(\varphi_k)\big\}
\end{multline*}

The cardinality of the cuspidal Asai transfer subfamily is
\[\ll \sum_{|t_k| \leq M} |\{\pi \in \mathfrak{F}_4(2DQ)\colon |L(\tfrac{1}{2},\pi \times \pi_{\varphi_k})|\geq C(\pi)^{\frac{1}{2}-\frac{\epsilon}{9 \cdot 10^{11}}} C(\pi_{\varphi_k})\}|.\]
By \cref{prop:subconvexity} and the Weyl law \eqref{eqn:Weyl_law}, this is $O_{D,\epsilon}(M^2 Q^{\epsilon})$.

Next, if $\As \phi$ is noncuspidal, \cref{lem:Asaifibres} shows that there exists a Hecke--Maa\ss{} newform $\varphi$ of arithmetic conductor $D$ and nebentypus $\chi_D$ such that
\[
L(\tfrac{1}{2}, \As \phi \times \varphi_k) = L(\tfrac{1}{2}, (\Ad \varphi \otimes \chi_D) \times \varphi_k) L(\tfrac{1}{2},\varphi_k),
\]
and this map from $\phi$ to $\Ad \varphi \otimes \chi_D$ is injective. Upon invoking the convexity bound 
$L(\tfrac{1}{2},\pi_{\varphi_k})\ll C(\pi_{\varphi_k})^{1/4}$ from \eqref{eqn:convexity}, we deduce that the cardinality of the noncuspidal transfer family is
\[\ll \sum_{|t_k|\leq M} |\{\pi \in \mathfrak{F}_3(2DQ)\colon |L(\tfrac{1}{2},\pi \times \pi_{\varphi_k})| \geq C(\pi)^{\frac{1}{2}-\frac{\epsilon}{9 \cdot 10^{11}}} C(\pi_{\varphi_k})^{\frac{3}{4}}\}|.\]
Again by \cref{prop:subconvexity} and the Weyl law \eqref{eqn:Weyl_law}, this is $O_{D,\epsilon}(M^2 Q^{\epsilon})$.
\end{proof}

\subsection{Proofs of Lemmata \ref{lem:realBS}, \ref{lem:realEis}, and \ref{lem:realtriple}}
\label{sect:proofsperiods}

The Fourier expansion of a Hilbert Hecke--Maa\ss{} cusp form $\phi$ is
\begin{equation}
\label{eqn:realFourier}
\phi(z_1,z_2) = \rho(\phi) \sum_{\substack{\alpha \in \cO_E \\ \alpha \neq 0}} \lambda_{\phi}(\alpha) \sqrt{y_1} K_{it_{1,\phi}}\Big(\frac{2\pi |\alpha| y_1}{\sqrt{D}}\Big) \sqrt{y_2} K_{it_{2,\phi}}\Big(\frac{2\pi |\sigma(\alpha)| y_2}{\sqrt{D}}\Big) e\Big(\frac{\alpha x_1 - \sigma(\alpha) x_2}{\sqrt{D}}\Big).
\end{equation}
The positive constant $\rho(\phi)$ ensures that $\phi$ is $L^2$-normalized with respect to the measure $d\mu$ on $\SL_2(\cO_E) \backslash \H \times \H$. An exact formula for $\rho(\phi)$ is given below.

\begin{lemma}[{\cite[Lemma 3.2]{Che21}}]
Let $\phi$ be a Hilbert Hecke--Maa\ss{} cusp form. Then
\begin{equation}
\label{eqn:rhophi}
\rho(\phi)^2 = \frac{8}{\sqrt{D}} \frac{1}{\Lambda(1, \Ad \phi)}.
\end{equation}
\end{lemma}

\begin{proof}
This is essentially proven in \cite[Lemma 3.2]{Che21}, albeit with some minor errata; we sketch the main ideas. First, we let $\Phi$ denote the ad\`{e}lic lift of $\phi$. Then by \cite[Proposition 6]{Wal85} and \cite[(3.16)]{Che21}, we have that
\begin{equation}
\label{eqn:realphiL2}
\int\limits_{\Zgp(\A_E) \GL_2(E) \backslash \GL_2(\A_E)} |\Phi(g)|^2 \, dg = \frac{1}{8} \rho(\phi)^2 \frac{\Lambda(1, \Ad \phi)}{\xi_E(2)}.
\end{equation}
Here $dg$ denotes the Tamagawa measure, so that $\Zgp(\A_E) \GL_2(E) \backslash \GL_2(\A_E)$ has volume $2$, and we have taken $q = 1$ in \cite[(3.16)]{Che21} and corrected the erroneous factor $2^{-2\delta_{\mathcal{D}}}$ to instead be $1/16$. Our result differs additionally from that in \cite[(3.16)]{Che21} since our definition \eqref{eqn:Lambdaspixpi'} of the completed $L$-function includes the arithmetic conductor and the discriminant. It remains to note that
\[\int\limits_{\Zgp(\A_E) \GL_2(E) \backslash \GL_2(\A_E)} |\Phi(g)|^2 \, dg = \frac{1}{\sqrt{D} \xi_E(2)} \int_{\SL_2(\cO_E) \backslash \H \times \H} |\phi(z_1,z_2)|^2 \, d\mu(z_1,z_2),\]
where the normalising factor comes from comparing the volume of $\Zgp(\A_E) \GL_2(E) \backslash \GL_2(\A_E)$ with respect to the Tamagawa measure to that of $\SL_2(\cO_E) \backslash \H \times \H$ with respect to $d\mu$.
\end{proof}

\begin{lemma}[{Cf.~\cite[Lemma 4.3]{Che21}}]
Let $\phi$ be a Hilbert Hecke--Maa\ss{} cusp form, and suppose that $\Re(s) > 1$. Then
\begin{equation}
\label{eqn:realphiE}
\int_{\Gamma \backslash \H} \phi(z,z) E(z,s) \, d\mu(z) = \frac{1}{4} \rho(\phi) \frac{\Lambda(s, \As \phi)}{\xi(2s)}.
\end{equation}
\end{lemma}

\begin{proof}
Consider the integral
\[\int_{\Gamma \backslash \H} \phi(z,z) E(z,s) \, d\mu(z).\]
By unfolding via the automorphy of $\phi(z,z)$ and then inserting the expansion \eqref{eqn:realFourier}, this equals
\[\rho(\phi) \sum_{\substack{\alpha \in \cO_E \\ \alpha \neq 0}} \lambda_{\phi}(\alpha) \int_{0}^{\infty} K_{it_{1,\phi}}\Big(\frac{2\pi |\alpha| y}{\sqrt{D}}\Big) K_{it_{2,\phi}}\Big(\frac{2\pi |\sigma(\alpha)| y}{\sqrt{D}}\Big) y^{s - 1} \, dy \int_{0}^{1} e\Big(\frac{(\alpha - \sigma(\alpha)) x}{\sqrt{D}}\Big) \, dx.\]
The integral over $x$ vanishes unless $\alpha = \sigma(\alpha)$ (so that $\alpha = m$ for some $m \in \Z \setminus \{0\}$), in which case the integral over $x$ is equal to $1$. The remaining integral over $y$ is equal to
\[\frac{D^{s/2}}{8 |m|^s} \frac{\Gamma_{\R}(s + it_{1,\phi} + it_{2,\phi}) \Gamma_{\R}(s + it_{1,\phi} - it_{2,\phi}) \Gamma_{\R}(s - it_{1,\phi} + it_{2,\phi}) \Gamma_{\R}(s - it_{1,\phi} - it_{2,\phi})}{\Gamma_{\R}(2s)}\]
by \cite[6.576.4]{GR15}, while just as in \cite[Theorem 2]{Asa77}, we have that
\[\sum_{m \in \Z-\{0\}} \frac{\lambda_{\phi}(m)}{|m|^s} = \frac{2 L(s, \As \phi)}{\zeta(2s)}.\]
The desired identity thereby follows.
\end{proof}

\begin{proof}[Proof of \cref{lem:realBS}]
Taking the residue of both sides of \eqref{eqn:realphiE} at $s = 1$, we see that
\[
\int_{\Gamma \backslash \H} \phi(z,z) \, d\mu(z) = \frac{1}{2} \rho(\phi) \Res_{s = 1} \Lambda(s, \As \phi).
\]
From \cref{lem:BC}, $\Lambda(s, \As \phi)$ has a pole at $s = 1$ if and only if $\phi$ is the base change of a nondihedral Hecke--Maa\ss{} cuspidal newform $\varphi$ of weight $0$, level $D$, and nebentypus $\chi_D$. If this is the case, then $\Lambda(s, \As \phi) = \Lambda(s, \Ad \varphi \otimes \chi_D) \xi(s)$, and consequently
\[
\int_{\Gamma \backslash \H} \phi(z,z) \, d\mu(z) = \frac{1}{2} \rho(\phi) \Lambda(1, \Ad \varphi \otimes \chi_D).
\]
Finally, we note that $\Lambda(s, \Ad \phi) = \Lambda(s, \Ad \varphi) \Lambda(s, \Ad \varphi \otimes \chi_D)$, and so from \eqref{eqn:rhophi}, we have that $\rho(\phi) = 2\sqrt{2} D^{-1/4} (\Lambda(1, \Ad \varphi) \Lambda(1, \Ad \varphi \otimes \chi_D))^{-1/2}$.
\end{proof}

\begin{proof}[Proof of \cref{lem:realEis}]
This follows from \eqref{eqn:realphiE} via analytic continuation.
\end{proof}

\begin{proof}[Proof of \cref{lem:realtriple}]
If $W_k = -1$, then the result follows upon making the change of variables $z \mapsto -\overline{z}$. Otherwise, we apply \cite[Theorem 5.6]{Che21}, which states that
\begin{multline*}
\Big|\int\limits_{\Zgp(\A_{\Q}) \GL_2(\Q) \backslash \GL_2(\A_{\Q})} \Phi(h,h) \Psi_k(h) \, dh\Big|^2 = \frac{1}{4} \frac{\xi_E(2)}{\xi(2)^2} \frac{\Lambda(\frac{1}{2}, \As \phi \times \varphi_k)}{\Lambda(1, \Ad \phi) \Lambda(1, \Ad \varphi_k)}	\\
\times \int\limits_{\Zgp(\A_E) \GL_2(E) \backslash \GL_2(\A_E)} |\Phi(g)|^2 \, dg \int\limits_{\Zgp(\A_{\Q}) \GL_2(\Q) \backslash \GL_2(\A_{\Q})} |\Psi_k(h)|^2 \, dh.
\end{multline*}
Here all measures involved are the Tamagawa measures, $\Phi$ denotes the ad\`{e}lic lift of $\phi$, $\Psi_k$ denotes the ad\`{e}lic lift of $\varphi_k$, and we have used \cite[Proposition 6.14]{Che21} to determine the local constants arising from the archimedean place. The left-hand side is equal to
\[
\frac{6}{\pi} \Big|\int_{\Gamma \backslash \H} \phi(z,z) \varphi_k(z) \, d\mu(z)\Big|^2,
\]
where the normalising factor comes from comparing the volume of $\Zgp(\A_{\Q}) \GL_2(\Q) \backslash \GL_2(\A_{\Q})$ with respect to the Tamagawa measure to that of $\Gamma \backslash \H$ with respect to $d\mu$, while via \eqref{eqn:realphiL2}, the right-hand side is equal to
\[\frac{3}{16 \pi} \rho(\phi)^2 \frac{\Lambda(\frac{1}{2}, \As \phi \times \varphi_k)}{\Lambda(1, \Ad \varphi_k)},\]
since
\[\int\limits_{\Zgp(\A_{\Q}) \GL_2(\Q) \backslash \GL_2(\A_{\Q})} |\Psi_k(h)|^2 \, dh = \frac{6}{\pi} \int_{\Gamma \backslash \H} |\varphi_k(z)|^2 \, d\mu(z) = \frac{6}{\pi}\]
as $\varphi_k$ is $L^2$-normalized. It remains to insert the identity \eqref{eqn:rhophi}.
\end{proof}

\subsection{Proof of \texorpdfstring{\cref{thm:nonsplitrealQE}}{Theorem \ref{thm:nonsplitrealQE}}}

Given $H \in C_c^{\infty}(\Gamma \backslash \H)$, we consider
\begin{equation}
\label{eqn:Hdmuj}
\mathcal{D}_j(H) \coloneqq \int_{\Gamma \backslash \H} H(z) \, d\mu_j(z) - \frac{3}{\pi} \mu_j(\Gamma \backslash \H) \int_{\Gamma \backslash \H} H(z) \, d\mu(z).
\end{equation}

\begin{lemma}
Let $\phi_j$ be a Hilbert Hecke--Maa\ss{} cusp form.  If $H \in C_c^{\infty}(\Gamma \backslash \H)$, then for any $0 < \epsilon' < \tfrac{1}{2}$,
\begin{multline}
\label{eqn:DjHbound}
\big|\mathcal{D}_j(H)\big|^2 \ll_{H,D,\epsilon'} C(\As \phi_j)^{-\frac{1}{2} + \epsilon'} \sum_{|t_k| \leq C(\As \phi_j)^{\epsilon'}} L(\tfrac{1}{2}, \As \phi_j \times \varphi_k)	\\
+ C(\As \phi_j)^{-\frac{1}{2} + \epsilon'} \int_{-C(\As \phi_j)^{\epsilon'}}^{C(\As \phi_j)^{\epsilon'}} \big|L(\tfrac{1}{2} + it, \As \phi_j)\big|^2 \, dt + C(\As \phi_j)^{-100}.
\end{multline}
\end{lemma}

\begin{proof}
By the spectral decomposition of $H$, Lemmata \ref{lem:realEis} and \ref{lem:realtriple}, and Stirling's formula, $|\mathcal{D}_j(H)|$ is
\begin{multline*}
\ll_D \sum_{\varphi_k} |\langle H, \varphi_k\rangle| \Big(\frac{L(\tfrac{1}{2}, \As \phi_j \times \varphi_k)}{L(1, \Ad \phi_j) L(1, \Ad \varphi_k)}\Big)^{\frac{1}{2}} e^{-\frac{\pi}{2} \Omega(t_k,t_{1,j},t_{2,j})} \prod_{\pm_1, \pm_2} (3 + |t_{1,j} \pm_1 t_{2,j} \pm_2 t_k|)^{-\frac{1}{4}}	\\
+ \int_{-\infty}^{\infty} |\langle H, E(\cdot, \tfrac{1}{2} + it)\rangle| \frac{|L(\tfrac{1}{2} + it, \As \phi_j)|}{L(1, \Ad \phi_j)^{\frac{1}{2}} |\zeta(1 + 2it)|} e^{-\frac{\pi}{2} \Omega(t,t_{1,j},t_{2,j})} \prod_{\pm_1, \pm_2} (3 + |t_{1,j} \pm_1 t_{2,j} \pm_2 t|)^{-\frac{1}{4}} \, dt,
\end{multline*}
where $\Omega(t,t_{1,j},t_{2,j})$ is as in \eqref{eqn:Omegadef}. Since $H$ is smooth and the Laplacian is self-adjoint, we have that $\langle H, \varphi_k\rangle = \lambda_k^{-N} \langle \Delta^N H, \varphi_k\rangle$ and $\langle H, E(\cdot,\tfrac{1}{2} + it)\rangle = (\frac{1}{4} + t^2)^{-N} \langle \Delta^N H, E(\cdot,\tfrac{1}{2} + it)\rangle$ for any nonnegative integer $N$. By the Cauchy--Schwarz inequality and Bessel's inequality, we deduce that for any nonnegative integer $N$,
\begin{multline*}
|\mathcal{D}_j(H)|^2 \ll_{D,N} \|\Delta^N H\|_2^2 \sum_{\varphi_k} \frac{L(\tfrac{1}{2}, \As \phi_j \times \varphi_k)}{L(1, \Ad \phi_j) L(1, \Ad \varphi_k)} (3 + |t_k|)^{-4N}	\\
\times e^{-\pi \Omega(t_k,t_{1,j},t_{2,j})} \prod_{\pm_1, \pm_2} (3 + |t_{1,j} \pm_1 t_{2,j} \pm_2 t_k|)^{-\frac{1}{2}}	\\
+ \|\Delta^N H\|_2^2 \int_{-\infty}^{\infty} \frac{|L(\tfrac{1}{2} + it, \As \phi_j)|^2}{L(1, \Ad \phi_j) |\zeta(1 + 2it)|^2} (3 + |t|)^{-4N}	\\
\times e^{-\pi \Omega(t,t_{1,j},t_{2,j})} \prod_{\pm_1, \pm_2} (3 + |t_{1,j} \pm_1 t_{2,j} \pm_2 t|)^{-\frac{1}{2}} \, dt.
\end{multline*}
Taking $N$ sufficiently large and invoking the convexity bound \eqref{eqn:convexity}, we see that we may truncate the sum over $\varphi_k$ to $|t_k| \leq C(\As \phi_j)^{\epsilon'}$ and the integral over $t$ to $|t| \leq C(\As \phi_j)^{\epsilon'}$ at the cost of an error term of size $O_{H,D}(C(\As \phi_j)^{-100})$. In these remaining ranges, we have that
\[\prod_{\pm_1, \pm_2} (3 + |t_{1,j} \pm_1 t_{2,j} \pm_2 t_k|)^{-\frac{1}{2}} \ll C(\As \phi_j)^{-\frac{1}{2}}.\]
The desired bound \eqref{eqn:DjHbound} for $|\mathcal{D}_j(H)|^2$ then follows from the bounds
\begin{gather*}
L(1,\Ad \phi_j)^{-1} \ll_{D,\epsilon'} C(\As \phi_j)^{\epsilon'}, \qquad L(1,\Ad \varphi_k)^{-1} \ll_{\epsilon'} C(\varphi_k)^{\epsilon'},	\\
|\zeta(1 + 2it)|^{-2} \ll_{\epsilon'} (3 + |t|)^{\epsilon'}.
\qedhere
\end{gather*}
\end{proof}

\begin{proof}[Proof of \cref{thm:nonsplitrealQE}]
Let $\mathscr{D}_1(Q,Q^{\epsilon'})$ and $\mathscr{D}_2(Q,Q^{\epsilon'})$ be as in \cref{prop:Asaibounds}. Let $\phi_j$ be an element of $\mathpzc{F}_{\As}(Q)$. If $\phi_j \notin \mathscr{D}_1(Q,Q^{\epsilon'})$, we have via \cref{prop:Asaibounds} that
\[C(\As \phi_j)^{-\frac{1}{2} + \epsilon'} \int_{-C(\As \phi_j)^{\epsilon'}}^{C(\As \phi_j)^{\epsilon'}} \big|L(\tfrac{1}{2} + it, \As \phi_j)\big|^2 \, dt \ll_{\epsilon'} C(\As \phi_j)^{- \frac{\epsilon}{9 \cdot 10^{11}} + 4\epsilon'}.\]
Similarly, if $\phi_j \notin \mathscr{D}_2(Q,Q^{\epsilon'})$, we have via \cref{prop:Asaibounds} and the Weyl law \eqref{eqn:Weyl_law} that
\[C(\As \phi_j)^{-\frac{1}{2} + \epsilon'} \sum_{|t_k| \leq C(\As \phi_j)^{\epsilon'}} L(\tfrac{1}{2}, \As \phi_j \times \varphi_k) \ll_{\epsilon'} C(\As \phi_j)^{- \frac{\epsilon}{9 \cdot 10^{11}} + 5\epsilon'}.\]
By \eqref{eqn:DjHbound}, we deduce that if $\phi_j \notin \mathscr{D}_1(Q,Q^{\epsilon'}) \cup \mathscr{D}_2(Q,Q^{\epsilon'})$, then
\[\mathcal{D}_j(H) \ll_{H,D,\epsilon'} C(\As \phi_j)^{- \frac{\epsilon}{9 \cdot 10^{11}} + 5\epsilon'}.\]
By \cref{prop:Asaibounds}, we have that $|\mathscr{D}_1(Q,Q^{\epsilon'})| + |\mathscr{D}_2(Q,Q^{\epsilon'})| \ll_{D,\epsilon'} Q^{\epsilon + 2\epsilon'}$. We finish by taking $\epsilon' = \tfrac{1}{5} \cdot 10^{-13} \epsilon$ and rescaling $\epsilon$.
\end{proof}

\subsection{Nonsplit quantum ergodicity for imaginary quadratic fields}
\label{sec:imaginary}

Finally, we consider the analogous problem in the setting of imaginary quadratic fields instead of real quadratic fields. Let $E = \Q(\sqrt{D})$ be an imaginary quadratic field with ring of integers $\cO_E$, where $D < 0$ is a fundamental discriminant; we assume for simplicity that $E$ has class number $1$. In place of $\H \times \H = (\SL_2(\R) \times \SL_2(\R)) / (\SO(2) \times \SO(2))$, we work on hyperbolic three-space $\H^3 = \SL_2(\mathbb{C}) / \SU(2)$, where we identity $\H^3$ with the subspace $\{P = x + iy + jr \colon x,r \in \R, \ y > 0\}$ of the Hamiltonian quaternions. A Bianchi Hecke--Maa\ss{} cusp form of level $\cO_E$ is an $L^2$-normalized smooth function $\phi \colon \H^3 \to \mathbb{C}$ for which
\begin{itemize}
\item $\phi$ is an eigenfunction of the weight $0$ Laplacian
\[\Delta \coloneqq -y^2 \Big(\frac{\partial^2}{\partial x^2} + \frac{\partial^2}{\partial y^2} + \frac{\partial^2}{\partial r^2}\Big) + y \frac{\partial}{\partial y},\]
so that $\Delta \phi(P) = \lambda_{\phi} \phi(P)$ for some $\lambda_{\phi} = 1 + 4t_{\phi}^2$ (and necessarily $t_{\phi} \in \R \cup -i[-\tfrac{7}{64},\tfrac{7}{64}]$),
\item $\phi$ is automorphic, so that $\phi(\gamma P) = \phi(P)$ for all $\gamma \in \SL_2(\cO_E)$, where
\[
\gamma P \coloneqq (aP + b) (cP + d)^{-1},\qquad \gamma = \begin{pmatrix} a & b \\ c & d \end{pmatrix}
\]
with the inverse and multiplication performed in the quaternion division algebra,
\item $\phi$ is of moderate growth,
\item $\phi$ is cuspidal, and
\item $\phi$ is a joint eigenfunction of every Hecke operator.
\end{itemize}
There is an embedding $\H \hookrightarrow \H^3$ given by the map $x + iy \mapsto x + iy$; we write $z$ for both the element $x + iy \in \H^2$ and for $x + iy \in \H^3$. A Bianchi Hecke--Maa\ss{} cusp form $\phi$ is $\SL_2(\Z)$-invariant when restricted to this embedding; thus $\phi(z)$ may be viewed as the restriction of a Bianchi Hecke--Maa\ss{} cusp form to the modular surface $\Gamma \backslash \H$.

The proof of \cref{thm:nonsplitimQE} is by the same methods as that of \cref{thm:nonsplitrealQE}; we therefore do not give details but rather highlight what alterations must be made. There are two major differences between the proofs of \cref{thm:nonsplitrealQE,thm:nonsplitimQE}. The first difference is that although the analogues of Lemmata \ref{lem:realBS}, \ref{lem:realEis}, and \ref{lem:realtriple} are valid in this setting, the bounds \eqref{eqn:realBSbounds} for $\mu_j(\Gamma \backslash \H)$ instead become
\[C(\phi_j)^{-1/8} \exp(-c_1 (\log C(\phi_j))^{1/2}) \ll_{D} \mu_j(\Gamma \backslash \H) \ll_{D} C(\phi_j)^{-1/8} \exp(c_2 (\log C(\phi_j))^{1/2})\]
(cf.~\cite[p.~2]{Hou}). This polynomial decay in $C(\phi) \coloneqq (3 + |t_{\phi}|)^4$ is why we must include the additional factor $C(\phi_j)^{1/8}$ in \eqref{eqn:imQUE}. This polynomial decay stems from the fact that in this setting, the square root of the gamma factors occurring in the completed $L$-functions on the right-hand side of \eqref{eqn:realBS} are
\[\sqrt{\frac{\Gamma_{\R}(2) \prod_{\pm} \Gamma_{\R}(2 \pm 2it_{\phi})}{\Gamma_{\R}(1) \prod_{\pm} \Gamma_{\R}(1 \pm 2it_{\phi})}},\]
and by Stirling's formula, this is asymptotic to $\frac{1}{\sqrt{2}\pi} (3 + |t_{\phi}|)^{-1/2} \asymp C(\phi)^{-1/8}$. The second difference is that the gamma factors present on the left-hand side of \eqref{eqn:Lindeloftoprove} are instead
\[\frac{\prod_{\pm_1, \pm_2} \Gamma_{\R}(\frac{1}{2} \pm_1 2it_{\phi} \pm_2 it_k) \prod_{\pm} \Gamma_{\mathbb{C}}(\frac{1}{2} \pm it_k)}{\Gamma_{\mathbb{C}}(1) \prod_{\pm} \Gamma_{\mathbb{C}}(1 \pm 2it_{\phi}) \Gamma_{\R}(1 \pm 2it_k)},\]
where $\Gamma_{\mathbb{C}}(s) \coloneqq 2(2\pi)^{-s} \Gamma(s)$. By Stirling's approximation, this is asymptotic to
\[4\pi^2 e^{-\pi \Omega(t_k,t_{\phi})} (3 + |t_{\phi}|)^{-1} \prod_{\pm} (3 + |2t_{\phi} \pm t_k|)^{-\frac{1}{2}}, \qquad \Omega(t,t_{\phi}) \coloneqq \begin{dcases*}
0 & if $|t| \leq 2|t_{\phi}|$,	\\
|t| - 2|t_{\phi}| & if $|t| \geq 2|t_{\phi}|$,
\end{dcases*}\]
while
\[C(\As \phi \times \varphi_k) = (3 + |t_k|)^4 (3 + |2t_{\phi} + t_k|)^2 (3 + |2t_{\phi} - t_k|)^2.\]
For this reason, showing that $|\mathcal{D}_j(H)| < C(\phi_j)^{-1/8 - \delta}$ for almost all $\phi_j \in \mathscr{F}(Q)$ essentially reduces to showing that $L(\frac{1}{2},\As \phi_j \times \phi_k) < Q^{1/4 - 2\delta}$ for almost all $\phi_j \in \mathscr{F}(Q)$.

\bibliographystyle{abbrv}
\bibliography{Humphries_Thorner_GLmxGLn_zero_density}

\def\cprime{$'$}
\begin{thebibliography}{10}

\bibitem{ABL18}
M.~Abert, N.~Bergeron, and E.~Le~Masson.
\newblock {Eigenfunctions and Random Waves in the Benjamini-Schramm limit}.
\newblock To appear in {\it {J}. {T}opol. {A}nal.}

\bibitem{Asa77}
T.~Asai.
\newblock On certain {D}irichlet series associated with {H}ilbert modular forms
  and {R}ankin's method.
\newblock {\em Math. Ann.}, 226(1):81--94, 1977.

\bibitem{Banks}
W.~D. Banks.
\newblock Twisted symmetric-square {$L$}-functions and the nonexistence of
  {S}iegel zeros on {${\rm GL}(3)$}.
\newblock {\em Duke Math. J.}, 87(2):343--353, 1997.

\bibitem{BB2}
V.~Blomer and F.~Brumley.
\newblock On the {R}amanujan conjecture over number fields.
\newblock {\em Ann. of Math. (2)}, 174(1):581--605, 2011.

\bibitem{Brumley}
F.~Brumley.
\newblock Effective multiplicity one on {${\rm GL}_N$} and narrow zero-free
  regions for {R}ankin-{S}elberg {$L$}-functions.
\newblock {\em Amer. J. Math.}, 128(6):1455--1474, 2006.

\bibitem{BM23}
F.~Brumley and J.~Matz.
\newblock Quantum ergodicity for compact quotients of {${\rm SL}_d(\Bbb R)/{\rm
  SO}(d)$} in the {B}enjamini-{S}chramm limit.
\newblock {\em J. Inst. Math. Jussieu}, 22(5):2075--2115, 2023.

\bibitem{BM}
F.~{Brumley} and D.~{Mili{\'c}evi{\'c}}.
\newblock {Counting cusp forms by analytic conductor}.
\newblock {\em Ann. Sci. {\'E}c. Norm. Sup{\'e}r. (4)}.
\newblock Accepted for publication.

\bibitem{BTZ}
F.~Brumley, J.~Thorner, and A.~Zaman.
\newblock Zeros of {R}ankin-{S}elberg {$L$}-functions at the edge of the
  critical strip.
\newblock {\em J. Eur. Math. Soc. (JEMS)}, 24(5):1471--1541, 2022.
\newblock With an appendix by Colin J. Bushnell and Guy Henniart.

\bibitem{BH}
C.~J. Bushnell and G.~Henniart.
\newblock An upper bound on conductors for pairs.
\newblock {\em J. Number Theory}, 65(2):183--196, 1997.

\bibitem{MR3647437}
J.~Buttcane and R.~Khan.
\newblock On the fourth moment of {H}ecke--{M}aass forms and the random wave
  conjecture.
\newblock {\em Compos. Math.}, 153(7):1479--1511, 2017.

\bibitem{Che21}
Y.~Cheng.
\newblock Special value formula for the twisted triple product {$L$}-function
  and an application to the restricted {$L^2$}-norm problem.
\newblock {\em Forum Math.}, 33(1):59--108, 2021.

\bibitem{MR818831}
Y.~Colin~de Verdi\`ere.
\newblock Ergodicit\'{e} et fonctions propres du laplacien.
\newblock {\em Comm. Math. Phys.}, 102(3):497--502, 1985.

\bibitem{GJ}
S.~Gelbart and H.~Jacquet.
\newblock A relation between automorphic representations of {${\rm GL}(2)$} and
  {${\rm GL}(3)$}.
\newblock {\em Ann. Sci. \'Ecole Norm. Sup. (4)}, 11(4):471--542, 1978.

\bibitem{GJ2}
R.~Godement and H.~Jacquet.
\newblock {\em Zeta functions of simple algebras}.
\newblock Lecture Notes in Mathematics, Vol. 260. Springer-Verlag, Berlin-New
  York, 1972.

\bibitem{GR15}
I.~S. Gradshteyn and I.~M. Ryzhik.
\newblock {\em Table of integrals, series, and products}.
\newblock Elsevier/Academic Press, Amsterdam, eighth edition, 2015.
\newblock Translated from the Russian. Translation edited and with a preface by
  Daniel Zwillinger and Victor Moll.

\bibitem{HL}
J.~Hoffstein and P.~Lockhart.
\newblock Coefficients of {M}aass forms and the {S}iegel zero.
\newblock {\em Ann. of Math. (2)}, 140(1):161--181, 1994.
\newblock With an appendix by Dorian Goldfeld, Jeffrey Hoffstein and Daniel
  Lieman.

\bibitem{Hou}
J.~{Hou}.
\newblock {Bounds for the periods of eigenfunctions on arithmetic hyperbolic
  3-manifolds over surfaces}.
\newblock {\em arXiv e-prints}, page arXiv:2304.04863, Apr. 2023.

\bibitem{Humphries_mult}
P.~Humphries.
\newblock Spectral multiplicity for {M}aa{\ss} newforms of non-squarefree
  level.
\newblock {\em Int. Math. Res. Not. IMRN}, 2019(18):5703--5743, 2019.

\bibitem{Humphries}
P.~Humphries.
\newblock Standard zero-free regions for {R}ankin-{S}elberg {$L$}-functions via
  sieve theory.
\newblock {\em Math. Z.}, 292(3--4):1105--1122, 2019.
\newblock With an appendix by Farrell Brumley.

\bibitem{HumphriesThorner_ZDE}
P.~Humphries and J.~Thorner.
\newblock Zeros of {R}ankin-{S}elberg {$L$}-functions in families.
\newblock {\em Compos. Math.}, 160(5):1041--1072, 2024.

\bibitem{Iwaniec_spectral}
H.~Iwaniec.
\newblock {\em Spectral methods of automorphic forms}, volume~53 of {\em
  Graduate Studies in Mathematics}.
\newblock American Mathematical Society, Providence, RI; Revista Matem\'{a}tica
  Iberoamericana, Madrid, second edition, 2002.

\bibitem{IK}
H.~Iwaniec and E.~Kowalski.
\newblock {\em Analytic number theory}, volume~53 of {\em American Mathematical
  Society Colloquium Publications}.
\newblock American Mathematical Society, Providence, RI, 2004.

\bibitem{JPSS}
H.~Jacquet, I.~I. Piatetskii-Shapiro, and J.~A. Shalika.
\newblock Rankin-{S}elberg convolutions.
\newblock {\em Amer. J. Math.}, 105(2):367--464, 1983.

\bibitem{Kim}
H.~H. Kim.
\newblock Functoriality for the exterior square of {${\rm GL}_4$} and the
  symmetric fourth of {${\rm GL}_2$}.
\newblock {\em J. Amer. Math. Soc.}, 16(1):139--183, 2003.
\newblock With appendix 1 by Dinakar Ramakrishnan and appendix 2 by Henry H.
  Kim and Peter Sarnak.

\bibitem{KMV}
E.~Kowalski, P.~Michel, and J.~VanderKam.
\newblock Rankin-{S}elberg {$L$}-functions in the level aspect.
\newblock {\em Duke Math. J.}, 114(1):123--191, 2002.

\bibitem{Kri03}
M.~Krishnamurthy.
\newblock The {A}sai transfer to {$\mathrm{GL}_4$} via the
  {L}anglands-{S}hahidi method.
\newblock {\em Int. Math. Res. Not. IMRN}, 2003(41):2221--2254, 2003.

\bibitem{Kri12}
M.~Krishnamurthy.
\newblock Determination of cusp forms on {$GL(2)$} by coefficients restricted
  to quadratic subfields (with an appendix by {D}ipendra {P}rasad and {D}inakar
  {R}amakrishnan).
\newblock {\em J. Number Theory}, 132(6):1359--1384, 2012.

\bibitem{Lan80}
R.~P. Langlands.
\newblock {\em Base change for {${\rm GL}(2)$}}, volume~96 of {\em Annals of
  Mathematics Studies}.
\newblock Princeton University Press, Princeton, N.J.; University of Tokyo
  Press, Tokyo, 1980.

\bibitem{Lap03}
E.~M. Lapid.
\newblock On the nonnegativity of {R}ankin-{S}elberg {$L$}-functions at the
  center of symmetry.
\newblock {\em Int. Math. Res. Not.}, 2003(2):65--75, 2003.

\bibitem{LS17}
E.~Le~Masson and T.~Sahlsten.
\newblock Quantum ergodicity and {B}enjamini-{S}chramm convergence of
  hyperbolic surfaces.
\newblock {\em Duke Math. J.}, 166(18):3425--3460, 2017.

\bibitem{LS24}
E.~Le~Masson and T.~Sahlsten.
\newblock Quantum ergodicity for {E}isenstein series on hyperbolic surfaces of
  large genus.
\newblock {\em Math. Ann.}, 389(1):845--898, 2024.

\bibitem{Li}
X.~Li.
\newblock Upper bounds on {$L$}-functions at the edge of the critical strip.
\newblock {\em Int. Math. Res. Not. IMRN}, 2010(4):727--755, 2010.

\bibitem{Lindenstrauss}
E.~Lindenstrauss.
\newblock Invariant measures and arithmetic quantum unique ergodicity.
\newblock {\em Ann. of Math. (2)}, 163(1):165--219, 2006.

\bibitem{LRS}
W.~Luo, Z.~Rudnick, and P.~Sarnak.
\newblock On the generalized {R}amanujan conjecture for {${\rm GL}(n)$}.
\newblock In {\em Automorphic forms, automorphic representations, and
  arithmetic ({F}ort {W}orth, {TX}, 1996)}, volume~66 of {\em Proc. Sympos.
  Pure Math.}, pages 301--310. Amer. Math. Soc., Providence, RI, 1999.

\bibitem{LS}
W.~Z. Luo and P.~Sarnak.
\newblock Quantum ergodicity of eigenfunctions on
  {$\mathrm{PSL}_2(\mathbf{Z})\backslash \mathbf{H}^2$}.
\newblock {\em Inst. Hautes \'{E}tudes Sci. Publ. Math.}, 81:207--237, 1995.

\bibitem{MW}
C.~M{\oe}glin and J.-L. Waldspurger.
\newblock Le spectre r\'esiduel de {${\rm GL}(n)$}.
\newblock {\em Ann. Sci. \'Ecole Norm. Sup. (4)}, 22(4):605--674, 1989.

\bibitem{MS}
W.~M{\"u}ller and B.~Speh.
\newblock Absolute convergence of the spectral side of the {A}rthur trace
  formula for {${\rm GL}_n$}.
\newblock {\em Geom. Funct. Anal.}, 14(1):58--93, 2004.
\newblock With an appendix by E. M. Lapid.

\bibitem{Nel20}
P.~D. Nelson.
\newblock {Quadratic Hecke Sums and Mass Equidistribution}.
\newblock {\em Int. Math. Res. Not. IMRN}, 05 2021.
\newblock rnab093.

\bibitem{NPS}
P.~D. Nelson, A.~Pitale, and A.~Saha.
\newblock Bounds for {R}ankin--{S}elberg integrals and quantum unique
  ergodicity for powerful levels.
\newblock {\em J. Amer. Math. Soc.}, 27(1):147--191, 2014.

\bibitem{Pet23}
C.~{Peterson}.
\newblock {Quantum ergodicity on the Bruhat-Tits building for
  $\mathrm{PGL}(3,F)$ in the Benjamini-Schramm limit}.
\newblock {\em arXiv e-prints}, page arXiv:2304.08641, Apr. 2023.

\bibitem{Rai17}
J.~Raimbault.
\newblock On the convergence of arithmetic orbifolds.
\newblock {\em Ann. Inst. Fourier (Grenoble)}, 67(6):2547--2596, 2017.

\bibitem{Ramakrishnan}
D.~Ramakrishnan.
\newblock Modularity of the {R}ankin-{S}elberg {$L$}-series, and multiplicity
  one for {${\rm SL}(2)$}.
\newblock {\em Ann. of Math. (2)}, 152(1):45--111, 2000.

\bibitem{RS}
Z.~Rudnick and P.~Sarnak.
\newblock The behaviour of eigenstates of arithmetic hyperbolic manifolds.
\newblock {\em Comm. Math. Phys.}, 161(1):195--213, 1994.

\bibitem{Sound_QUE_Maass}
K.~Soundararajan.
\newblock Quantum unique ergodicity for {${\rm SL}_2(\mathbb{Z})\backslash
  \mathbb{H}$}.
\newblock {\em Ann. of Math. (2)}, 172(2):1529--1538, 2010.

\bibitem{ST}
K.~Soundararajan and J.~Thorner.
\newblock Weak subconvexity without a {R}amanujan hypothesis.
\newblock {\em Duke Math. J.}, 168:1231--1268, 2019.
\newblock With an appendix by Farrell Brumley.

\bibitem{Ste94}
G.~Steil.
\newblock Eigenvalues of the {L}aplacian and of the {H}ecke operators for
  $\mathrm{PSL}(2,\mathbb{Z})$.
\newblock Technical Report DESY 94-028, Hamburg, 1994.

\bibitem{MR0402834}
A.~I. \v{S}nirel'man.
\newblock Ergodic properties of eigenfunctions.
\newblock {\em Uspehi Mat. Nauk}, 29(6(180)):181--182, 1974.

\bibitem{Wal85}
J.-L. Waldspurger.
\newblock Sur les valeurs de certaines fonctions {$L$} automorphes en leur
  centre de sym\'{e}trie.
\newblock {\em Compositio Math.}, 54(2):173--242, 1985.

\bibitem{Watson}
T.~C. Watson.
\newblock {\em Rankin triple products and quantum chaos}.
\newblock ProQuest LLC, Ann Arbor, MI, 2002.
\newblock Thesis (Ph.D.)--Princeton University.

\bibitem{You16}
M.~P. Young.
\newblock The quantum unique ergodicity conjecture for thin sets.
\newblock {\em Adv. Math.}, 286:958--1016, 2016.

\bibitem{MR916129}
S.~Zelditch.
\newblock Uniform distribution of eigenfunctions on compact hyperbolic
  surfaces.
\newblock {\em Duke Math. J.}, 55(4):919--941, 1987.

\bibitem{MR1101972}
S.~Zelditch.
\newblock Selberg trace formulae and equidistribution theorems for closed
  geodesics and {L}aplace eigenfunctions: finite area surfaces.
\newblock {\em Mem. Amer. Math. Soc.}, 96(465):vi+102, 1992.

\bibitem{Zel10}
S.~Zelditch.
\newblock Recent developments in mathematical quantum chaos.
\newblock In {\em Current developments in mathematics, 2009}, pages 115--204.
  Int. Press, Somerville, MA, 2010.

\end{thebibliography}

\end{document}